\documentclass[reqno,10pt]{amsart}

\usepackage{amsmath, amsthm, amssymb, amsfonts}
\usepackage[pagebackref,colorlinks,citecolor=blue,linkcolor=blue,urlcolor=blue,filecolor=blue]{hyperref}
\usepackage{aliascnt}
\usepackage{shuffle}
\usepackage{blkarray}
\usepackage{graphicx}
\usepackage[all]{xy}
\usepackage{color}
\usepackage{mathtools}
\usepackage{ulem}
 \calclayout

\numberwithin{thmcounter}{section}
\newaliascnt{thmauto}{thmcounter}

\newaliascnt{Defauto}{thmcounter}

\newaliascnt{exauto}{thmcounter}

\newaliascnt{lemauto}{thmcounter}

\newaliascnt{propauto}{thmcounter}

\newaliascnt{conjauto}{thmcounter}

\newaliascnt{corauto}{thmcounter}

\newaliascnt{remauto}{thmcounter}

\newaliascnt{consauto}{thmcounter}

\newaliascnt{notationauto}{thmcounter}

\newaliascnt{convauto}{thmcounter}

\newtheorem{athm}{Theorem}

\newtheorem{thm}[thmauto]{Theorem}
\newtheorem{lem}[lemauto]{Lemma}
\newtheorem{prop}[propauto]{Proposition}
\newtheorem{cor}[corauto]{Corollary}

\theoremstyle{remark}
\newtheorem{rem}[remauto]{Remark}

\theoremstyle{definition}
\newtheorem{defn}[Defauto]{Definition}
\newtheorem{ex}[exauto]{Example}

\newtheorem{notation}[notationauto]{Notation}

\newcommand{\FI}{\mathsf{FI}}
\newcommand{\FJ}{\mathsf{FJ}}

\newcommand{\OI}{\mathsf{OI}}

\newcommand{\Ab}{\mathsf{Ab}}

\newcommand{\C}{\mathcal C}
\newcommand{\D}{\mathcal D}

\newcommand{\FIopMod}{\mathrm{Mod}_{\FI^{\op}}}
\newcommand{\FIopmod}{\FIopMod}
\newcommand{\xmod}[1]{\mathrm{Mod}_{#1}}
\newcommand{\Catalan}{\mathrm{Catalan}}
\newcommand{\CB}{\mathsf{CB}}

\newcommand{\cb}{\pi}

\newcommand{\N}{\mathbb N}
\newcommand{\Z}{\mathbb Z}

\newcommand{\Q}{\mathbb Q}
\newcommand{\R}{\mathbb R}
\newcommand{\SG}{\mathfrak S}
\newcommand{\inject}{\hookrightarrow}
\newcommand{\surject}{\twoheadrightarrow}
\newcommand{\op}{\mathrm{op}}
\newcommand{\eps}{\varepsilon}
\newcommand{\X}{\chi}

\DeclareMathOperator{\id}{id}
\DeclareMathOperator{\Emb}{Emb}

\DeclareMathOperator{\colim}{colim}
\DeclareMathOperator{\Hom}{Hom}

\DeclareMathOperator{\im}{im}
\DeclareMathOperator{\coker}{coker}

\DeclareMathOperator{\Mat}{Mat}
\DeclareMathOperator{\res}{res}

\newcommand{\fn}[1]{\vphantom{\scalebox{1.3}{\mbox{\fbox{\tt{1}}}}}\raisebox{.3pt}{\scalebox{.7}{\mbox{\fbox{\tt{#1}}}}}}

\renewcommand{\fn}[1]{#1}

\title{On the tails of $\FI$-modules}
\author{Peter Patzt \and John D. Wiltshire-Gordon}
\date{\today}

\begin{document}
\begin{abstract}
We study the end-behavior of integer-valued $\FI$-modules.
Our first result describes the high degrees of an $\FI$-module in terms of newly defined tail invariants.
Our main result provides an equivalence of categories between $\FI$-tails and finitely supported modules for a new category that we call $\FJ$.  Objects of $\FJ$ are natural numbers, and morphisms are  infinite series with summands drawn from certain modules of Lie brackets.  
\end{abstract}

\maketitle


\section{Introduction}
Let $\FI$ be the category of finite sets and injections, and $R$ a commutative ring.  An \emph{$R$-valued $\FI$-module} is a functor $\FI \to \xmod R$. By default, we work with $R = \Z$.

If $M$ is a finitely generated $\Q$-valued $\FI$-module, Snowden \cite{Snowden13} proved that the Hilbert function
\begin{equation} \label{eq:dims}
n \mapsto \mathrm{dim} \, M_n
\end{equation}
eventually coincides with a polynomial. Shortly thereafter, Church--Ellenberg--Farb \cite{CEF15} showed that the decomposition of $M_n$ into irreducible $\SG_n$-representations satisfies a stability pattern termed ``multiplicity stability'' by Church--Farb \cite{CF13}.  Subsequently, \cite{CEFN14} proved that the eventual polynomiality of \eqref{eq:dims} holds over any field. 
Following these foundational results, the study of $\FI$-modules has seen notable successes including \cite{Nagpal15, CE17, LiRamos18, Ramos17, NSS17, Harman16} and applications including \cite{Calegari15,CP15,PX16,KM18,GL,Tosteson16}.

An $\FI$-module is \emph{presented in degrees $\le d$} if 
one can find generators and relations in degrees $\le d$.
We emphasize that the number of generators and relations may be infinite. Nevertheless, 
multiplicity stability still holds for $\Q$-valued $\FI$-modules presented in finite degree, where the stable multiplicities may be infinite. And for $R$-valued $\FI$-modules, presentation in finite degree is equivalent to a polynomiality notion of Dwyer \cite{Dwyer80}; see \autoref{defn:poly} and \cite[Theorem 3.30]{MPW}.

Our first theorem describes the high degrees of $R$-valued $\FI$-modules. It is driven by a new combinatorial basis of $\Z\FI(d,n)$ for $n\ge 2d-1$ that we call the Catalan Basis; see \autoref{sec:CB} and especially \autoref{cor:X}.

\begin{athm} \label{thm:main_concrete}
If $M$ is an $R$-valued $\FI$-module presented in in degrees $\leq d$, then there are $R$-modules $A_0, \dots, A_d$ such that there is an isomorphism of $R$-modules
\[
M_n \cong \bigoplus_{\ell=0}^{d} A_{\ell}^{\oplus \binom n\ell - \binom n{\ell-1}}
\]
for all $n \geq 2d-1$.
\end{athm}

Even for field-valued $\FI$-modules, \autoref{thm:main_concrete} is new, as it guarantees the Hilbert polynomial expands nonnegatively in the $\Z$-basis of integer-valued polynomials $\{\binom{n}{\ell} - \binom{n}{\ell-1}\}_{\ell \in \mathbb{N}}$.  
\autoref{thm:main} will show that the $\FI$-action on the right-hand-side is determined by natural maps between the $R$-modules $A_0, \dots, A_d$ of \autoref{thm:main_concrete}.

\subsection{The tail invariants of an $\FI$-module}
The abelian groups $A_0, A_1, \dots$ from \autoref{thm:main_concrete} are functorial in the $\FI$-module $M$, and moreover, they depend only on the tail of $M$.  

We call $A_\ell$ the \emph{$\ell$-th tail invariant} of $M$. It can be computed by tensoring with a certain flat $\FI^\op$-module $\Xi(\ell)$ that we define momentarily:
\[
A_\ell \cong M \otimes_{\FI} \Xi(\ell).
\]
The symbol $\otimes_{\FI}$ denotes the functor tensor product, which is the quotient of $\bigoplus_n M_n \otimes \Xi(\ell)_n$ by the relation $m f \otimes \xi - m \otimes f \xi$ for all $\FI$-morphisms $f$, and elements $m \in M$, $\xi \in \Xi(\ell)$; see \autoref{sec:tensor-hom}.

\begin{defn}\label{defn:Xi(ell)}
Let $\Xi(\ell)$ denote the $\FI^\op$-module
\[ \Xi(\ell)_S = H_0(X_S, A_S),\]
where $X_S = \Emb(S, [\ell] \sqcup \mathbb{R})$ is a space of embeddings, and $A_S$ is the subspace of those $\phi \colon S \to [\ell] \sqcup \mathbb{R}$ with $[\ell] \not\subseteq \im \phi$. Injections act on $\Xi(\ell)$ by precomposition. 
\end{defn}

\begin{notation}\label{rem:x}
Fix an infinite increasing sequence  $x_1 < x_2 < \ldots$ of real numbers. Every homology class $\xi \in \Xi(\ell)_n$ is uniquely represented as 
a linear combination of embeddings $\phi \colon [n] \to [\ell] \sqcup \{x_1, \dots, x_{n-\ell}\}$. If we write such a $\phi$ in one-line notation, it denotes the homology class of $\phi$ in $\Xi(\ell)$.
\end{notation}

\begin{ex}
Writing functions $[n] \to [\ell] \sqcup \R$ in one-line notation, we have
\[
\begin{array}{cclcccccccc}
\Xi(0)_3 &\cong& \mathbb{Z} \cdot \{ & x_1x_2x_3, & x_1x_3x_2, &x_2x_1x_3, &x_2x_3x_1, &x_3x_1x_2, &x_3x_2x_1 & \} \\
\Xi(1)_3 &\cong& \mathbb{Z} \cdot \{ &1x_1x_2, &1x_2x_1, &x_11x_2, &x_1x_21, &x_21x_1, &x_2x_11 & \} \\
\Xi(2)_3 &\cong& \mathbb{Z} \cdot \{ &12x_1, &1x_12, &21x_1, &2x_11, &x_112, &x_121& \} \\
\Xi(3)_3 &\cong &\mathbb{Z} \cdot \{ &123,&132,&213,&231,&312,&321 &\}
\end{array}
\]
and $\Xi(\ell)_3 \cong 0$ for $\ell\ge 4$.
\end{ex}

\begin{rem}
The zeroth tail invariant of an $\FI$-module $M$ is given by the formula $\colim_{\OI} M$, where $\OI$ is the subcategory of $\FI$ consisting of the sets $[n]= \{1, \dots, n\}$ for all $n\in \N$ and order-preserving injections.  That the functor $M \mapsto \colim_{\OI} M$ is exact is a result of Isbell--Mitchell and Isbell \cite{IM73, Isbell74}.
\end{rem}

\subsection{The category of $\FI$-tails}

Define the category of \emph{R-valued $\FI$-tails} to be the Serre quotient
\[
 \frac{\{\mbox{$R$-valued $\FI$-modules presented in finite degree}\}}{\{\mbox{$R$-valued $\FI$-modules supported in finite degree}\}},
\]
whose objects correspond to end-behaviors of $\FI$-modules.  This category detects representation stability phenomena. 

\begin{defn}\label{defn:FJ}
Let $\FJ$ be the category whose objects are $\N = \{0,1,2,\dots\}$ and 
\[\FJ(\ell,m) = \Hom_{\FI^{\op}}(\Xi(\ell),\Xi(m)).\] 
\end{defn}

We give an inverse limit description of   $\FJ(\ell,m)$ in \autoref{sec:defFJ}, and 
a combinatorial description  as infinite series with summands expressed using Lie brackets and shuffle products  in \autoref{sec:combFJ}. We note that $\FJ$ has upwards and downwards maps. Indeed, $\FJ(\ell,m)$ is a torsion-free abelian group of infinite rank for all $\ell, m\in \N$. We also remark that $\FJ$ is a $\Z$-linear category, and it always acts linearly.

By construction, 
the family of $\FI^\op$-modules $\Xi(-)$ carries an action of $\FJ$ that commutes with the action of $\FI^\op$, and so defines an $\FI^\op \times \FJ$-module.

\begin{athm} \label{thm:main}
The functor
\[ M \longmapsto M \otimes_\FI \Xi(-)\]
induces an equivalence from the category of $R$-valued $\FI$-tails to the category of $R$-valued $\FJ$-modules supported in finite degree.
\end{athm}

\autoref{thm:main} is the integral generalization of the following surprising result of Sam--Snowden.
\begin{thm}[{{\cite[Theorem 2.5.1]{SS16}}}] \label{thm:ssrational}
The category of $\mathbb Q$-valued $\FI$-tails is equivalent to the category of $\mathbb Q$-valued $\FI$-modules supported in finite degree.
\end{thm}

\subsection{Polynomial degree}

Recall the following definition originally due to Dwyer \cite{Dwyer80}; this concrete definition coincides with Djament--Vespa's \cite{DjamentVespa19} ``weak polynomial degree'' for $\FI$-modules presented in finite degree.

\begin{defn}\label{defn:poly}
An $\FI$-module $M$ that is presented in finite degree is said to have \emph{polynomial degree $\le 0$} if it is eventually constant, and \emph{polynomial degree $\leq d$} if $\coker (M \to \Sigma M)$ has polynomial degree $\leq (d-1)$, where $\Sigma M = M([1] \sqcup -)$ is the shift. 
\end{defn}

We remark that the assumption that $M$ is presented in finite degree implies that $\ker(M\to \Sigma M)$ is supported in finite degree (see for example \cite[Theorem 3.30]{MPW}). Over a field, an $\FI$-module has polynomial degree $\le d$ if and only if the degree of the Hilbert polynomial \eqref{eq:dims} is at most $d$.

To formulate a version of  \autoref{thm:main} for tails  of $\FI$-modules of polynomial degree $\le d$, let us define the following category.

\begin{defn}\label{defn:FJd}
Let $\FJ_{\le d}$ be the category whose objects are $ \{0,1,2,\dots,d\}$ and 
\[\FJ_{\le d}(\ell,m) = \Hom_{\FI^{\op}_{\le d}}(i^*_d\Xi(\ell),i^*_d\Xi(m)),\] 
where $i^*_d$ denotes the inclusion of  the full subcategory $\FI_{\le d} \subset \FI$ of all sets with cardinality at most $d$.
\end{defn}

It turns out that for $\ell,m \le d$ we get \[\FJ_{\leq d}(\ell, m) \cong \FJ(\ell, m) / I_d(\ell,m),\] where $I_d(\ell,m)$ is spanned by composites $\ell \to d+1 \to m$; see \autoref{prop:factoring}.  Every $\FJ$-module supported in degrees $\{0, \dots, d\}$ is automatically an $\FJ_{\leq d}$-module, because $I_{d}(\ell,m)$ already acts by zero. Vice versa, an $\FJ_{\le d}$-module can be extended to an $\FJ$-module by zeros. By this logic, we consider $\xmod{\FJ_{\le d}}$ to be a subcategory of $\xmod{\FJ}$.

\begin{athm} \label{thm:degrees}
Under the equivalence in \autoref{thm:main}, the subcategory of  
$\FI$-tails of polynomial degree $\leq d$ corresponds to the subcategory of $\FJ_{\leq d}$-modules. 
\end{athm}

The following corollary follows from \autoref{prop:XiKan}, which says we can understand $\FJ_{\le d}(\ell,m)$ as a subset $ \Z\SG_d$. 

\begin{cor} \label{cor:Morita}
There exists a $(d+1)\times(d+1)$ matrix ring $Q_d \subseteq \Mat_{d+1}(\Z \SG_d)$ whose category of right modules is equivalent to the category of $\FI$-tails of polynomial degree $\leq d$. 
\end{cor}

\begin{ex}
The following $Q_0, Q_1, Q_2$ are examples for \autoref{cor:Morita}.
\begin{itemize}
\item $Q_0 = \Mat_1(\mathbb{Z} \SG_0) \cong \Z$. (This simply means that the tails of eventually-constant $\FI$-modules are in bijection with abelian groups.)
\item $Q_1 = \begin{bmatrix} \Z \cdot \fn{1} & \Z \cdot \fn{1} \\ 0 & \Z \cdot \fn{1} \end{bmatrix} \subset \Mat_2(\mathbb{Z} \SG_1)$.
\item $Q_2 = \begin{bmatrix} \mathbb{Z} \cdot \fn{12} + \mathbb{Z} \cdot \fn{21} & \mathbb{Z} \cdot (\fn{12}+\fn{21}) & \mathbb{Z} \cdot \fn{12} + \mathbb{Z} \cdot \fn{21} \\  \mathbb{Z} \cdot (\fn{12}-\fn{21}) & \mathbb{Z} \cdot \fn{12} & \mathbb{Z} \cdot \fn{12} + \mathbb{Z} \cdot \fn{21} \\  \mathbb{Z} \cdot (\fn{12} - \fn{21}) & 0 & \mathbb{Z} \cdot \fn{12} + \mathbb{Z} \cdot \fn{21} \end{bmatrix} \subset \Mat_3(\mathbb{Z} \SG_2)$, 
which is a subring with $\mathbb{Z}$-rank $12$.  So a quadratic tail can be encoded as a single abelian group with $12$ compatible endomorphisms.
\end{itemize}
\end{ex}

\subsection{Computing the tail invariants from a presentation matrix} \label{sec:compute}
Any finitely presented $\FI$-module is the cokernel of a map between free $\FI$-modules, and this map is described by a presentation matrix.  If $M$ is generated in degrees $a_1, \ldots, a_g$ and related in degrees $b_1, \ldots, b_r$, then the presentation matrix takes the form
\[
\begin{blockarray}{ccccc}
 & b_1 & b_2 & \cdots & b_r \\
\begin{block}{c[cccc]}
 a_1  &  &  & & \\
 a_2  &  &  & & \\
\vdots & & & & \\
 a_g  & & & & \\
\end{block}\end{blockarray}\vspace{-1.5ex}\]
where the entry in position $(i,j)$ is a formal $\mathbb{Z}$-linear combination of injections $[a_i] \to [b_j]$.  We call such a matrix an \emph{$\FI$-matrix}.  

\begin{athm} \label{thm:compute}
If $Z$ is a presentation matrix for an $\FI$-module $M$, and if $W$ is the $\FJ$-module corresponding to the tail of $M$, then
\[
W_\ell \cong \mathrm{coker} \, \Xi(\ell)_Z
\]
where $\Xi(\ell)_Z$ denotes the integer block matrix obtained by evaluating the module $\Xi(\ell)$ at the entries of the $\FI$-matrix $Z$.
\end{athm}
 
A version of  \autoref{thm:compute} over $\mathbb{Q}$ appears in \cite{WG18}, which relies on the structure theory provided by \cite{SS16}.  An analogous result for the category $\mathsf{FA}$ of finite sets and all functions---and for any other category of dimension zero---is available in the second author's dissertation \cite{WG16, WG15}. 

In the following examples, we continue to write injections using one-line notation.
\begin{ex} \label{ex:first}
Let $M$ be the $\FI$-module spanned by symbols $z_{ij}$ for all $i \neq j$, subject to the relation $z_{ij} + z_{jk}+z_{ki}=0$.  A presentation matrix for $M$ is given by
\[
Z = \begin{blockarray}{cc}
 & 3  \\
\begin{block}{c[c]}
 2  & \fn{12} + \fn{23} + \fn{31}  \\
\end{block}
\end{blockarray}.\vspace{-1.5ex}
\]
If $\ell >2$, then $\Xi(\ell)_Z=0$.  For $\ell \leq 2$, we have
\begin{align*}
\Xi(0)_Z &= \begin{blockarray}{ccccccc}
& x_1x_2x_3 & x_1x_3x_2 &x_2x_1x_3 &x_2x_3x_1 &x_3x_1x_2 &x_3x_2x_1  \\
\begin{block}{c[cccccc]}
x_1x_2  &  2 & 1 & 1 & 2 & 2 & 1 \\ 
x_2x_1  &  1 & 2 & 2 & 1 & 1 & 2 \\
\end{block}
\end{blockarray}\\
\Xi(1)_Z &= \begin{blockarray}{ccccccc}
&1x_1x_2 &1x_2x_1 &x_11x_2 &x_1x_21 &x_21x_1 &x_2x_11  \\
\begin{block}{c[cccccc]}
1x_1  &  1 & 1 & 1 & 1 & 1 & 1 \\ 
x_11 &  1 & 1 & 1 & 1 & 1 & 1 \\
\end{block}
\end{blockarray} \\
\Xi(2)_Z &= \begin{blockarray}{ccccccc}
&12x_1 &1x_12 &21x_1 &2x_11 &x_112 &x_121  \\
\begin{block}{c[cccccc]}
12  &  1 & 0 & 0 & 1 & 1 & 0 \\
21  &  0 & 1 & 1 & 0 & 0 & 1 \\
\end{block}
\end{blockarray}. 
\end{align*}

Since $\mathrm{coker}\, \Xi(0)_Z \cong \mathbb{Z}/3$, $\mathrm{coker}\, \Xi(1)_Z \cong \mathbb{Z}$, and $\mathrm{coker}\, \Xi(2)_Z \cong 0$, we have
\[
M_n \cong (\mathbb{Z}/3) \oplus \mathbb{Z}^{n-1}
\]
for all $n \geq 5$, by \autoref{thm:compute} and \autoref{thm:main_concrete}.  If the relation defining $M$ were instead $2z_{ij} + 3z_{ik} + 4z_{ji} + 5z_{jk} + 6z_{ki} + 7z_{kj}=0$, an equally-straightforward calculation finds
\[
M_n \cong (\mathbb{Z}/27) \oplus (\mathbb{Z}/45)^{n-1} \oplus (\mathbb{Z}/3)^{n(n-3)/2}
\]
for all $n \geq 5$.
\end{ex}

\subsection{Acknowledgements}
 The authors wish to thank Jordan Ellenberg and Jeremy Miller for many useful conversations. Further thanks to Zach Himes and Rohit Nagpal for helpful comments. Thanks also to MSRI for its hospitality. This second author was supported by the Algebra RTG at the University of Wisconsin, DMS-1502553.


\section{Notation and category theory background}

We quickly review the (mostly standard) notation and the background needed in category theory that we use in this paper.

\textbf{Set-theory notation.} Let $\N$ denote the set of nonnegative integers $\{0,1,2, \dots\}$ and $[n]$ denote the set $\{1,\dots, n\}$. If $f\colon X\to Y$ is a function between the sets $X$ and $Y$, and $S\subseteq X$, then $f|_S\colon S \to Y$ denotes the restriction. If $T \subseteq Y$ such that $\im f\subseteq T$, we write $f|^T \colon X \to T$ by limiting the codomain. If $f(S) \subseteq T$, we write $f|_S^T\colon S \to T$ limiting both the domain and the codomain.

\textbf{Algebraic notation.} For a set $X$, let $\Z X$ denote the free abelian group whose basis is $X$. An unadorned tensor product $M \otimes N$ means $M \otimes_\Z N$. 
The symmetric group of bijections $[n] \to [n]$ is denoted by $\SG_n$.

\textbf{Categorical notation.} Let $\C$ be a locally small category. We write $c \in \C$ to mean that $c$ is an object in $\C$ and we write $\C(c,c')$ for the set of morphisms between the objects $c,c'\in \C$. If $\D$ is another category, and $F \colon \C \to \D$ a functor (by which we always mean a covariant functor), then we write $F_c$ or $Fc$ for the action of $F$ on an object $c\in\C$, and $F_f$ or $Ff$ for the action of $F$ on a morphism $f\in \C(c,c')$.

Suppose $\C$ is essentially small.  If $F,F' \colon \C \to \Ab$ are functors, then $\Hom_\C(F,F')$ denotes the set of natural transformations from $F$ to $F'$. We write $\xmod\C$ for the locally small category of functors $\C \to \Ab$. If $\C$ is a linear category, we require the functors $\C \to \Ab$ to be linear.
The resulting covariant action of $\C$ is written on the right; correspondingly, contravariant actions are written on the left.  
In detail, if $M \in \xmod\C$, $c\in C$, $m \in M_c$, and $f\in \C(c,c')$, we simply write $mf$ or $m \cdot f$ for $M_f(m)\in M_{c'}$.  Similarly, if $M \in \xmod{\C^\op}$, $c\in C$, $m \in M_c$, and $f\in \C(c',c)$, we write $fm$ or $f \cdot m$ for $M_f(m)\in M_{c'}$.  In this context, we write $ff'$ or $f\cdot f'$ for the composition $f'\circ f$ if $f\in\C(c,c')$ and $f'\in\C(c',c'')$ to conform with our right-action convention. 
If $M,N\in \xmod\C$, the set of natural transformations $\Hom_\C(M,N)$ carries the structure of an abelian group.

\subsection{The tensor-hom adjunction over categories}\label{sec:tensor-hom}

We recall the tensor-hom adjunction for modules over a category. For the remainder of this section, fix an essentially small category $\C$, and let $i \colon \C' \subseteq \C$ be the inclusion of a small skeleton so that $i$ is an equivalence.   

We recall the functor tensor product over $\C$.

\begin{defn}
Let $M$ be a $\C$-module and $N$ a $\C^\op$-module. The tensor product $M\otimes_\C N \in \Ab$ is the cokernel of the homomorphism
\[   \bigoplus_{c,c' \in \C'} M_c \otimes \Z\C'(c,c') \otimes N_{c'} \longrightarrow \bigoplus_{c \in \C'} M_c \otimes N_c \]
given by $m\otimes f \otimes n \mapsto  mf \otimes n - m\otimes fn$.
\end{defn}
When $N\colon \C^\op \times \D \to \Ab$ carries two actions---one covariant, and one contra---we call it a bimodule, and write the actions on opposite sides.  This works because morphisms of $\C^\op$ commute with morphisms of $\D$ in the product category $\C^\op \times \D$.
\begin{thm}[Tensor-hom adjunction]\label{thm:tensorhom}
If $N\colon \C^\op \times \D \to \Ab$ is a bimodule, then $- \otimes_\C N \colon \xmod\C \to \xmod\D$ is the left adjoint of $\Hom_\D(N,-) \colon \xmod\D \to \xmod\C$. In particular,
\[ \Hom_\D(M\otimes_\C N, A) \cong \Hom_{\C}(M,\Hom_\D(N,A)).\]
Denoting by $\varphi_m \colon \Z\C(c,-) \to M$ the map corresponding by Yoneda's lemma to $m \in M_c$ for $c \in \C$, we have the following formula for the unit $\eta$:
\begin{align*}
\eta_M \colon M &\to \Hom_{\D}(N,M\otimes_\C N) \\
m & \mapsto \varphi_m.
\end{align*}
\end{thm}

\begin{cor}\label{cor:F*leftadjoint}
Let $G\colon \C \to \D$, then \[- \otimes_\C \Z\D(G(-),-)\colon \xmod\C \to \xmod{\D}\] is the left adjoint of the precomposition functor $G^*\colon \xmod\D \to \xmod\C$.
\end{cor}

\begin{proof}
By Yoneda's lemma, $\Hom_\D(\Z\D(Gc,-),M) \cong M(Gc) \cong (G^*M)_c$, and these isomorphisms are natural in $c \in \C$ and $M \in \xmod{\D}$.
\end{proof}

\begin{cor}\label{cor:F*rightadjoint}
Let $F\colon \D \to \C$, then $F^*\colon \xmod\C \to \xmod\D$ is the left adjoint of \[\Hom_\D(\Z\C(-,F(-)),-)\colon \xmod\D \to \xmod\C.\] The unit $\eta_M \colon M \to \Hom_\D(\Z\C(-,F(-)),F^*M)$ is given by sending $m\in M_c$ to the map that sends $\phi\in \Z\C(c,F(d))$ to $ m\phi \in (F^*M)_d$.
\end{cor}

\begin{proof}
The first claim follows from Yoneda's lemma.  For the second, using the description of the unit from \autoref{thm:tensorhom}, we have that that $\eta_M$ sends $m\in M_c$ to $\phi_m \otimes_\C N\colon \Z\C(c,-)\otimes_\C N \to M\otimes_\C N$ with $\phi_m(\id_c) = m$. For $N = \Z\C(-,F(-)) \colon \C^\op \to \xmod\D$, $\phi\in \Z\C(c,F(d))$ corresponds to $\id_c \otimes \phi \in  \Z\C(c,-)\otimes_\C N$ and thus is sent to $m \otimes_\C \phi\in M \otimes \Z\C(-,F(-))$ which corresponds to $m\phi \in M_{F(d)} = (F^*M)_d$.
\end{proof}


\section{Tail invariants}
 In this section, we construct certain flat $\FI^{\op}$-modules $\Xi(\ell)$, yielding under tensor product exact functors $\xmod\FI \to \Ab$.  We will see that these functors send modules supported in finite degree to zero, and so they depend only on tails.

\subsection{The shift functor $\Sigma^T$ {of $\FI^\op$-modules} and its right adjoint $\Omega^T$}

\begin{defn}
Let $\Sigma^T\colon \FIopmod \to \FIopmod$ be the \emph{shift functor}, which is given by precomposition with the opposite of the functor $\sigma^T = T\sqcup-\colon \FI \to \FI$. For brevity, let us write $\Sigma$ for $\Sigma^{\{1\}}$.
\end{defn}

In the following proposition, we describe a functor $\Omega^T$ and show that it is the right adjoint of $\Sigma^T$. Write $\Omega^\ell$ for $\Omega^{[\ell]}$ and $\Omega$ for $\Omega^1$.

\begin{prop} \label{prop:Omegap}
The functor $\Sigma^T$ has a right adjoint $\Omega^T\colon \FIopmod \to \FIopmod$ so that
\[ (\Omega^TM)_S \cong \bigoplus_{f \in \FI(T,S)} M_{S\setminus \im f},\]
and for $g \in \FI(S',S)$,
\begin{align*}
(\Omega^T M)_g \colon\bigoplus_{f \in \FI(T,S)} M_{S\setminus \im f} &\longrightarrow \bigoplus_{f' \in \FI(T,S')} M_{S'\setminus \im f'} \\
(f,m) & \longmapsto \begin{cases} (f',g|^{S\setminus \im f}_{S'\setminus \im f'} \cdot m) & g\circ f' = f\\ 0 & \text{otherwise}.\end{cases}
\end{align*}

Given a map of $\FI^\op$-modules $\varphi\colon M\to M'$, {the induced map} $(\Omega^T)_\varphi$ is given by the sum of $\varphi_{S\setminus \im f} \colon M_{S\setminus \im f} \to M'_{S\setminus \im f}$ over all $f\in \FI(T,S)$.

The component of the unit $\eta$ corresponding to $M$ 
\[\eta_M\colon M_S \to (\Omega^T \Sigma^T M)_S \cong \bigoplus_{f\in\FI(T,S)} M_{T \sqcup S\setminus \im f}\]
is given by the sum of the maps $M_h \colon M_S \to M_{T \sqcup S\setminus \im f}$, where  $h|_T = f$ and $h$ restricted to $S \setminus \im f$ is the inclusion map $S\setminus \im f \subseteq S$.
\end{prop}

\begin{proof}
An injection $h \colon T\sqcup U \to S$ is determined uniquely by an injection $f\colon T \to S$ and an injection $h' \colon U \to S \setminus \im f$. Therefore, using \autoref{cor:F*rightadjoint}, we can compute:
\begin{align*}
(\Omega^T M)_S 
= &\Hom_{\FI^\op}(\Z\FI(T\sqcup -,S),M)\\
= &\Hom_{\FI^\op}(\bigoplus_{f\in\FI(T,S)} \Z\FI( -,S\setminus \im f),M)\\
\cong &\bigoplus_{f\in\FI(T,S)}  \Hom_{\FI^\op}(\Z\FI( -,S\setminus \im f),M)\\
\cong  & \bigoplus_{f \in \FI(T,S)} M_{S \setminus \im f}.
\end{align*}

To understand $(\Omega^TM)_g$ for $g\in \FI(S',S)$, let us go through these isomorphisms. 
An element
\[ (f,m) \in \bigoplus_{f \in \FI(T,S)} M_{S \setminus \im f}\]
translates to the map 
$\varphi \in \Hom_{\FI^\op}(\Z\FI(T\sqcup -,S),M)$
 with
\begin{align*}
 \varphi_U \colon \Z\FI(T\sqcup U, S) &\longrightarrow M_U\\ h  \hspace{2em}&\longmapsto \begin{cases}  h|^{S\setminus \im f}_U \cdot m &h|_T =f\\ 0 &\text{otherwise.}\end{cases}\end{align*}
Therefore $g\varphi_U \colon \Z\FI(T\sqcup U, S')  \to M_U$ is given by 
\begin{align*}h'  \longmapsto  \begin{cases}  (g\circ h')|^{S\setminus \im f}_U \cdot m = h'|^{S' \setminus g^{-1}(\im f)}_U \cdot g|^{S\setminus \im f}_{h'(U)} \cdot m &(g \circ h')|_T =f\\ 0 &\text{otherwise,}\end{cases}\end{align*}
and this map corresponds through the isomorphism to the sum
\[ \sum_{\substack{f'\in \FI(T,S')\\g \circ f' = f}}(f' , g|^{S\setminus \im f}_{S'\setminus \im f'} \cdot m) \in \bigoplus_{f' \in \FI(T,S')} M_{S' \setminus \im f'}.\]

Given a map of $\FI^\op$-modules $\varphi\colon M\to M'$, then $(\Omega^T)_\varphi$ is given by postcomposition
\[ \Hom_{\FI^\op}(\Z\FI(T\sqcup -,S),M) \longrightarrow \Hom_{\FI^\op}(\Z\FI(T\sqcup -,S),M').\]
Through the given isomorphisms, this translates to the sum
\[ \sum_{f\in\FI(T,S)} \varphi_{S\setminus \im f} \colon \bigoplus_{f \in \FI(T,S)} M_{S \setminus \im f} \longrightarrow \bigoplus_{f \in \FI(T,S)} M'_{S \setminus \im f}.\]

By \autoref{cor:F*rightadjoint}, the unit $\eta_M \colon M\to \Omega^T\Sigma^T M= \Hom_{\FI^\op}(\Z\FI(T\sqcup -, -), \Sigma^T M)$ sends an element $m \in M_{S}$ to the map of $\FI^\op$-modules which sends a morphism  $g \in \FI(T\sqcup U, S)$ to $gm \in M_{T\sqcup U} = (\Sigma^TM)_U$. Through the above isomorphism, this translates to $m\in M_S$ being sent to the sum
\[ \sum_{f\in \FI(T,S)} hm \in \bigoplus_{f\in\FI(T,S)} M_{T \sqcup (S\setminus \im f)}\]
where $h\in \FI(T\sqcup (S\setminus \im f), S)$ is given by $h|_T = f$ and $h|_{S \setminus \im f}$ is the inclusion map $S\setminus \im f \subset S$.
\end{proof}

\begin{prop}\label{prop:Xi(ell)description}
$\Xi(\ell) \cong \Omega^\ell\Xi(0)$
\end{prop}

\begin{proof}
By \autoref{prop:Omegap}, we have
\[ (\Omega^\ell\Xi(0))_S = \bigoplus_{f \in \FI([\ell],S)} H_0(\Emb(S\setminus \im f,\R)).\]
There is an isomorphism
\[\bigoplus_{f \in \FI([\ell],S)} H_0(\Emb(S\setminus \im f,\R)) \longrightarrow H_0(X_S,A_S) = \Xi(\ell)_S \]
with $X_S = \Emb(S, [\ell] \sqcup \R)$ and $A_S = \{ \phi \in X_S \mid [\ell] \not \subseteq \im \phi\}$ given by
\[(f, [\phi]) \longmapsto \left[s \mapsto \begin{cases} \phi(s)  &s \in S\setminus \im f\\f^{-1}(s) &s\in \im f\end{cases}\right].\]

To prove that these isomorphisms are natural with respect to the $\FI^\op$-action, let $g\in \FI(S',S)$. By \autoref{prop:Omegap}, $g$ sends $(f,[\phi])$ to the unique $(f',[\phi'])$ with $f = g \circ f'$ and $\phi' =\phi\circ g|_{S'\setminus \im f'}^{S\setminus \im f}$ if such exists, and to zero otherwise. Under the isomorphism,  $(f',[\phi'])$ maps to 
\[\psi'(s') = \begin{cases} \phi'(s')  &s' \in S'\setminus \im f'\\(f')^{-1}(s') &s'\in \im f'.\end{cases}\]
Going around the other way,  the isomorphism sends  $(f,[\phi])$ to 
\[\psi(s) = \begin{cases} \phi(s)  &s \in S\setminus \im f\\f^{-1}(s) &s\in \im f.\end{cases}\]
And $g$ maps $[\psi]$ to $[\psi \circ g]$ if $\psi^{-1}([\ell]) \subseteq \im g$, and to zero otherwise. Note that $\psi^{-1}([\ell]) = \im f$, and if $\im f\subseteq \im g$, there exists a unique $f'$ such that $f = g\circ f'$. This implies that the necessary square commutes for all $s' \in S'$:
\[ \psi \circ g(s') = \begin{cases} \phi(g(s'))  &g(s') \in S\setminus \im f\\f^{-1}(g(s')) &g(s')\in \im f\end{cases} = \begin{cases} \phi'(s')  &s' \in S'\setminus \im f'\\(f')^{-1}(s') &s'\in \im f'\end{cases} = \psi'(s').\qedhere\]
\end{proof}

In \autoref{sec:CB}, we will  reconstruct the  tail of an $\FI$-module $M$ from  knowledge of the tail invariants $M \otimes_\FI \Xi(\ell)$ and the homomorphisms induced by maps $\Xi(\ell) \to \Xi(m)$, which we will discuss in the next sections.

\subsection{Left Kan extension from $\OI^\op$ to $\FI^\op$} 

Here, we consider the inclusions $\OI^\op \subset \FI^\op$ and $\OI^\op_{\leq d} \subset \FI^\op_{\leq d}$. We may restrict $\FI^\op$-modules and $\FI^\op_{\le d}$-modules along these inclusions. In this section, we describe the left adjoints to these restrictions, each of which has a description as a tensor product by \autoref{cor:F*leftadjoint}.

\begin{prop} \label{prop:lan}
The tensor product
\[ \Z\FI(n,-)\otimes_\OI M \cong \Z\SG_n\otimes  M_n \]
and the action of $g \in \FI(n', n)$ on the left-hand side translates under this isomorphism to 
\begin{align*}
 \Z\SG_n \otimes M_n&\longrightarrow  \Z\SG_{n'} \otimes M_{n'}\\
\sigma \otimes m & \longmapsto \begin{cases} \tau  \otimes hm & \mbox{if $h= \tau^{-1}g\sigma$ is monotone}\\ 0 & \text{otherwise,}\end{cases}
\end{align*}
for $\sigma \in \SG_n$ and $\tau \in \SG_{n'}$.
\end{prop}

\begin{proof}
Observe that every map in $f \in \FI(n,k)$ can be uniquely written as $\sigma h$ for some $\sigma \in \SG_n$ and $h\in \OI(n,k)$.
Therefore $f \otimes m = \sigma \otimes hm$ for all $m \in M_k$, and this is a unique representative in $\Z\FI(n,n) \otimes M_n = \Z\SG_n \otimes M_n$.
This proves that
\[ \Z\FI(n,-) \otimes_\OI M \cong \Z\SG_n \otimes M_n.\]

Precomposing with $g\in \FI(n',n)$ will send $\sigma\otimes m$ to $g\sigma \otimes m=\tau \otimes hm$ if $\tau h = g\sigma$ for the unique $\tau \in \SG_{n'}$ and $h \in \OI(n',n)$.
\end{proof}

\begin{rem} \label{rem:leqd}
The same result and proof hold for the inclusions $\OI_{\le d} \subset \FI_{\le d}$, and so $\Z\FI_{\le d}(n,-) \otimes_{\OI_{\le d}} M \cong \Z\SG_n \otimes M_n$ for $n\le d$.
\end{rem}

\begin{defn}
Let $\Lambda(\ell)$ denote the $\OI^{\op}$-module with \[\Lambda(\ell)_n = \begin{cases} \Z &n\ge \ell\\0 &n<\ell\end{cases}\] and where $f\in \OI(n',n)$ acts by the identity whenever $[\ell] \subseteq \im f$, and acts by zero otherwise.  Since $f$ is monotone, this condition is equivalent to $f(i) = i$ for all $i \in [\ell]$.
\end{defn}

Recall that $x_1 < x_2 < \cdots$ is the arbitrary increasing sequence of real numbers from \autoref{rem:x}.

\begin{defn}\label{defn:xinell}
For $n\ge \ell$, let $\xi_{n,\ell} \in \Xi(\ell)_n$ denote the element
\[ \xi_{n,\ell} = 1 \dots \ell x_1 \dots x_{n-\ell}.\]
\end{defn}

\begin{prop}\label{prop:XiKan}
There is an isomorphism of $\FI^\op$-modules 
\[\Z\FI(-,-)\otimes_{\OI}\Lambda(\ell)\xrightarrow{\cong} \Xi(\ell)\]
that sends $\id_n \otimes 1 \in \Z\FI(n,-) \otimes \Lambda(\ell)$ to $\xi_{n,\ell} \in \Xi(\ell)_n$ if $n\ge \ell$. This isomorphism restricts to an isomorphism of $\FI_{\le d}^\op$-modules 
\[ \Z\FI_{\le d}(-,-)\otimes_{\OI_{\le d}}i^*_d\Lambda(\ell) \xrightarrow{\cong} i_d^* \Xi(\ell),\]
where $i_d$ denotes the inclusion $\FI_{\le d} \subset \FI$.
\end{prop}

\begin{proof}
According to \autoref{prop:lan}, the tensor product
\[ \Z\FI(n,-)\otimes_{\OI}\Lambda(\ell) = \begin{cases} \Z\SG_n\otimes \Z & n\ge \ell \\ 0 & n < \ell.\end{cases}\]
 It suffices to define the map in degrees $n\ge \ell$:
 \begin{align*} \Z\SG_n \otimes \Z&\longrightarrow \Xi(\ell)_n\\\sigma\otimes 1&\longmapsto \sigma\xi_{n,\ell}.\end{align*}
  This gives an isomorphism in every degree. It only remains to show that this is a map of $\FI^\op$-modules.

If $g \in \FI(n', n)$, then for every $\sigma \in \SG_n$, there is a unique pair $(\tau, h) \in \SG_{n'} \times \OI(n',n)$ such that $\tau g = h\sigma$.
 We want to check that the action of $g$ commutes with the isomorphism $\Z\SG_n \otimes \Lambda(\ell)_n \to \Xi(\ell)_n$ given above. Because $\sigma\otimes 1$ for $\sigma \in \SG_n$ gives a basis of $ \Z\SG_n \otimes \Lambda(\ell)_n$ for $n\ge \ell$, it is enough to check {commutation} on these elements. By \autoref{prop:lan}, $g$ sends $\sigma\otimes 1$ to $ \tau\otimes h(1)$, where 
\[h(1) = \begin{cases}1&[\ell] \subseteq \im h\\0&\text{otherwise.}\end{cases}\]
 On the other hand, $g$ sends $\sigma \xi_{n,\ell}$ to $\tau \xi_{n',\ell}$ if $[\ell] \subseteq \im (g\sigma^{-1})$ and to zero otherwise. This proves the first assertion because $\im h = \im (g\sigma^{-1})$.

The second assertion  follows from \autoref{rem:leqd}.
\end{proof}

\subsection{Natural transformations between tail invariants}\label{sec:defFJ}

\begin{prop}\label{prop:mapseq}
Let $\ell,m\in \N$. Then $\Hom_{\FI^{\op}}(\Xi(\ell),\Xi(m))$ is isomorphic to  the set of  $(\xi_n)_{n\ge \ell} \in \prod_{n\ge \ell} \Xi(m)_n $ such that 
\[\delta_i \xi_n = \begin{cases} \xi_{n-1}&i \not\in [\ell]\\0&\text{otherwise.}\end{cases}\]

Furthermore, for every $d\in\N$ the map 
\begin{align*}\Hom_{\FI^{\op}_{\le d}}(i^*_d\Xi(\ell),i^*_d\Xi(m))  &\longrightarrow  \Xi(m)_d\\\varphi &\longmapsto \varphi_d(\xi_{d,m})\end{align*}
  is injective, and the composition
\[ \Hom_{\FI^{\op}}(\Xi(\ell),\Xi(m)) \xrightarrow{i^*_d(-)} \Hom_{\FI^{\op}_{\le d}}(i^*_d\Xi(\ell),i^*_d\Xi(m)) \inject \Xi(m)_d\]
is given by $(\xi_n)_{n\in\N} \mapsto \xi_d$.
\end{prop}

\begin{proof}
By the tensor-hom adjunction and \autoref{prop:XiKan}, 
\[ \Hom_{\FI^{\op}}(\Xi(\ell),\Xi(m)) \cong \Hom_{\OI^{\op}}(\Lambda(\ell), \Xi(m)).\]
An element $\vartheta \in \Hom_{\OI^\op}(\Lambda(\ell),\Xi(m))$ can then be described by its values $\xi_n= \vartheta_n(1)$ in every degree $n\ge \ell$.  These values have  the property $\delta_i \xi_n = \xi_{n-1}$ if $i \in [n] \setminus [\ell]$ and $\delta_i \xi_n =0$ otherwise, coming from the definition of $\Lambda(\ell)$.

For the second assertion, it is enough to show that $\xi_n$ for $m\le n<d$ is determined by $\xi_d$. In fact, then $\xi_n = f\xi_d$, where $f$ is the inclusion $[n] \subset [d]$.
\end{proof}

\begin{rem}
For every $M \in \FIopmod$, the inclusion $\emptyset \sqcup S \subseteq T \sqcup S$ induces a map $M_{T \sqcup S} \to  M_{S}$, and so $M$ carries a natural map $\Sigma^T M \to M$.
\end{rem}

\begin{defn}
Let $\ell \in \N$ and consider the map $\Sigma\Xi(\ell) \to \Xi(\ell)$. The adjunction of $\Sigma$ and $\Omega$ defines a homomorphism $\eta_{\ell}\colon \Xi(\ell) \to \Omega \Xi(\ell) \cong \Xi(\ell+1)$.
\end{defn}

\begin{lem} \label{lem:adjoint_kernel}
Let $M \in \FIopmod$, and write $\omega \colon M \to \Omega M$ for the map that is adjoint to the natural map $\sigma\colon \Sigma M \to M$.  Then, $\ker(\omega)_n$ consists of those $m \in M_n$ with $gm = 0$ for all $g\in \OI(n-1,n)$.
\end{lem}

\begin{proof}
The formula
\[ \omega_{n} = (\Omega( \sigma) \circ \eta_M)_{n} \]
describes $\omega$, where $\eta$ is the unit of the $\Sigma \dashv \Omega$ adjunction described in \autoref{prop:Omegap}. Thus
\[\omega_n(m) = \Omega( \sigma) ( \eta(m)) = \Omega( \sigma)\Bigg( \sum_{f \in \FI (1,n)} hm \Bigg) =  \sum_{f \in \FI (1,n)} h|_{[n]\setminus \im f}\cdot m \in \bigoplus_{f\in\FI(1,n)} M_{[n]\setminus \im f} ,\]
where $h\in \FI([1] \sqcup ([n]\setminus \im f),[n])$ with $h|_{[1]} = f$ and $h|_{[n]\setminus \im f}$ is the inclusion $[n]\setminus \im f \subset [n]$. Note that this map can also be expressed as
\[ \omega_n(m) =  \sum_{g \in \OI (n-1,n)} gm\in \bigoplus_{g \in \OI (n-1,n)} M_{n-1} \cong (\Omega M)_n\]
because the condition $h\circ f = \id_{[1]}$ is automatic as there is only one map in $\FI(1,1)$ and $g = h|_{[n]\setminus \im f}$ is a monotone map.

In order for $m\in M_n$ to be in the kernel of $\omega_n$, the sum $\omega_n(m)$ has to vanish in every summand. This proves the lemma.
\end{proof}

\begin{defn}
Define the $\FI^{\op}$-module $D=\ker(\eta_{0}\colon \Xi(0) \to \Xi(1))$.
\end{defn}

\begin{rem}
In \autoref{sec:combFJ}, we will give a combinatorial basis for $D_n$ and discuss how it relates to Lie brackets and derangements. In particular, $D_n$ is a free abelian group whose rank is the number of derangements in $\SG_n$.
\end{rem}

\begin{lem}\label{cor:ell_step}
$\Omega^\ell D \cong \ker( \eta_{\ell}\colon \Xi(\ell) \to \Xi(\ell+1))$
\end{lem}
\begin{proof}
We will prove that the isomorphism from \autoref{prop:Xi(ell)description} restricts to the desired isomorphism. Let $\xi \in \Xi(\ell)_n$ correspond to 
\[\sum_{f\in \FI(\ell,n)} \xi_f \in \bigoplus_{f\in\FI(\ell,n)} \Xi(0)_{[n] \setminus \im f}.\] Then by \autoref{lem:adjoint_kernel}, $\xi \in  \ker( \eta_{\ell}\colon \Xi(\ell) \to \Xi(\ell+1))$ if and only if $g\xi = 0$ for all $g\in \OI(n-1,n)$. Thus the image of $ \ker( \eta_{\ell}\colon \Xi(\ell) \to \Xi(\ell+1))$ in $\Omega^\ell \Xi(0)$ is those $\sum_{f\in \FI(\ell,n)} \xi_f $ such that 
\begin{multline*} 0  = g\left(\sum_{f\in \FI(\ell,n)} \xi_f \right) = \sum_{f\in \FI(\ell,n)}\sum_{\substack{ f' \in \FI(\ell,n-1)\\ g\circ f' = f}} g|^{[n]\setminus \im f}_{[n-1]\setminus \im f'} \cdot \xi_f\\= \sum_{f'\in \FI(\ell,n-1)} g|^{[n]\setminus \im (g\circ f')}_{[n-1]\setminus \im f'} \cdot \xi_{g\circ f'} \in \bigoplus_{f'\in \FI(\ell,n-1)} \Xi(0)_{[n-1]\setminus \im f'}\end{multline*}
by \autoref{prop:Omegap}.
The restrictions $g|^{[n]\setminus \im f}_{[n-1]\setminus \im f'}$ are still monotone, so if $\xi_f\in D_{[n]\setminus \im f}$, each term in this sum vanishes. This shows that $\Omega^\ell D$ is contained in the image of $\ker( \eta_{\ell}\colon \Xi(\ell) \to \Xi(\ell+1))$ under the isomorphism. Next we want to see that this image is contained in $\Omega^\ell D$. Every term
\[g|^{[n]\setminus \im (g\circ f')}_{[n-1]\setminus \im f'} \cdot \xi_{g\circ f'}\]
in this sum has to be zero. By strategically choosing $g$ and $f'$, we can satisfy the vanishing condition of \autoref{lem:adjoint_kernel} for every $\xi_f$.
Indeed, for every $f\in \FI(\ell,n)$ and monotone injection  $g' \colon [n-\ell-1] \to [n] \setminus \im f$, choose $f' \in \FI(\ell,n-1)$ and $g \in \OI(n-1,n)$ such that $f = g \circ f'$ and the diagram
\[ \xymatrix{
[n-\ell-1] \ar[rrrd]^{g'} \ar[d]^\cong \\
[n-1] \setminus \im f' \ar[rrr]^{g|^{[n]\setminus \im f}_{[n-1]\setminus \im f'}} &&& [n] \setminus \im f
}\]
commutes, where the left map is the unique monotone bijection.
\end{proof}

\begin{lem}\label{lem:XiellXid}
Fix $d\ge \ell$. There is a map of $\FI^\op$-modules $u_{\ell,d}\colon \Xi(\ell) \to \Xi(d)$ that sends $\xi_{d,\ell}$ to $\xi_{d,d}$ in degree $d$.
\end{lem}

\begin{proof}
Write $W(d, q)$ for the set of words in the alphabet $\{x, y\}$ that are permutations of $x^dy^q$.  Let $w=w_1 \dots w_{d+q} \in W(d, q)$, and construct an embedding $\phi_w \colon [d+q] \to [d] \sqcup \R$ so that, for all $i \in [d+q]$,
\[w_i = x  \implies \phi_w(i) \in \R \quad \text{and} \quad w_i = y  \implies \phi_w(i) \in [d],\]
and moreover, if $w_i = w_j$ for $i < j$, then $\phi_w(i) < \phi_w(j)$.  In other words, $\phi_w$ is determined by requiring that it be a monotonic map to $\R$ or $[d]$ after restricting its domain to the positions of $x$ or $y$. 

Set $p, q \in \N$ so that $d + p = n$ and $\ell + q = d$.  We now argue that the assignment
\[
\xi_n = \sum_{w \; \in \; y^\ell \cdot W(p, q)} [\phi_{w}]
\]
satisfies the hypothesis of \autoref{prop:mapseq}, giving a map $\Xi(\ell) \to \Xi(d)$; note that $\xi_d = \xi_{d,d}$, because $p=q=0$ in this case, so $\xi_{d,\ell} \mapsto \xi_{d,d}$, as required.

Recall that the class of an embedding $\phi \colon [n] \hookrightarrow [d] \sqcup \R$ with $[d] \not \subseteq \im \phi$ vanishes in $\Xi(d)$.  By design, however, the letter $y$ appears exactly $d$ times in each element of $y^\ell \cdot W(p, q)$, so the terms of $\xi_n$ are nonzero.

If $i \in [\ell]$, then $i \not \in \im(\phi_w \circ \delta_i)$ for $w \in y^\ell \cdot W(p, q)$, because of the prefix $y^\ell$.  Consequently, $\delta_i \xi_n = 0 \in \Xi(d)_{n-1}$.  On the other hand, if $i > \ell$, then the sum splits
\[
\delta_i \xi_n   = \sum_{\substack{w \in y^\ell \cdot W(p, q) \\ w_i = x}} [\phi_{ w} \circ \delta_i] \quad + \quad \sum_{\substack{w \in y^\ell \cdot W(p, q) \\ w_i = y}} [\phi_{ w} \circ \delta_i].
\]
If $w_i = y$, then $\phi_{w}(i) \in [d]$, and so $[d] \not \subseteq \im(\phi_{w} \circ f)$, proving that these terms vanish in $\Xi(d)$.  On the other hand, if $w_i = x$, then $\phi_{w} \circ f$ is isotopic to $\phi_{w'}$ where $w' = w_1 \dots \widehat{w_i} \dots w_{n}$, since any two monotone embeddings $\phi^{-1}(\R) \to \R$ are related by a straight-line isotopy.  The words in $W(p,q)$ that have $x$ in position $i$ are in bijection with $W(p-1, q)$ by deleting this $x$.  Therefore, we have
\[
\xi_n \circ f  = \sum_{\substack{w \in y^\ell \cdot W(p, q) \\ w_j = x}} [\phi_{w} \circ f] =  \sum_{\substack{w \in  y^\ell \cdot W(p, q) \\ w_j = x}} [\phi_{w'} ] = \sum_{w' \in  y^\ell \cdot W(p-1, q)} [\phi_{w'} ] = \xi_{n-1}. \qedhere
\]
\end{proof}

\begin{thm} \label{thm:FJ}
Let $\ell,m\in \N$, then there is a tower
\[ \Hom_{\FI^{\op}}(\Xi(\ell),\Xi(m)) \surject \dots \surject  \Hom_{\FI_{\le d}^{\op}}(i^*_d\Xi(\ell),i_d^*\Xi(m))  \surject \dots \surject  \Hom_{\FI_{\le 0}^{\op}}(i^*_0\Xi(\ell),i_0^*\Xi(m)). \]
If $d\ge \ell,m$, the factors are given by the short exact sequence
\[ 0 \longrightarrow (\Omega^mD)_{d} \longrightarrow \Hom_{\FI^{\op}_{\le d}}(i^*_d\Xi(\ell),i^*_d\Xi(m)) \longrightarrow \Hom_{\FI^{\op}_{\le d-1}}(i^*_{d-1}\Xi(\ell),i^*_{d-1}\Xi(m)) \longrightarrow 0.\]
Otherwise, $i^*_d\Xi(\ell)$ or $i^*_d\Xi(m)$ is zero. 
\end{thm}

\begin{proof}
We organize the proof in the following way. We first show that $(\Omega^mD)_{d}$ is {isomorphic to} the kernel in the short exact sequence. Next we prove that elements in this kernel can be extended to $\Hom_{\FI^\op}(\Xi(\ell),\Xi(m))$ if $\ell = d$, and subsequently for all $\ell<d$. This shows that every map in a factor of the tower can be extended to the top. Therefore every map extends, and thus, all restriction maps are surjective.

Under the isomorphism from \autoref{prop:mapseq}, a morphism in $\ker(\res^d_{d-1})$ corresponds to an element $\xi \in \Xi(m)_d$. However, since this morphism is in the kernel of $\res^d_{d-1}$, it vanishes in degree $d-1$, which shows that $f\xi=0$ for all $f\in \OI(d-1,d)$. On the hand, any such $\xi$ describes a map in $\ker(\res^d_{d-1})$ by \autoref{prop:mapseq}. The subgroup of such elements is isomorphic to $(\Omega^mD)_{d}$ using \autoref{lem:adjoint_kernel} and \autoref{cor:ell_step}.

Next, we show that the maps in $\ker(\res^d_{d-1})$ are in the image of the restriction
\[ \res_{d} \colon  \Hom_{\FI^{\op}}(\Xi(\ell),\Xi(m))  \longrightarrow \Hom_{\FI^{\op}_{\le d}}(i^*_d\Xi(\ell),i^*_d\Xi(m)).\]
Let $\xi_d\in \Xi(m)_d$ correspond to some map $\varphi \in \ker(\res^d_{d-1})$ that we want to extend.  Being in the kernel implies that $f \xi_d = 0$ for every $f \in \OI(d-1, d)$.  We start with $\ell = d$.  Set $\xi_{d+n}\in \Xi(m)_{d+n}$ to be the concatenation of $\xi_d$ and the word $ x_{d-m+1} \dots x_{d-m+n}$. Then, for every $f'\in \OI(d+n',d+n)$, \[f' \xi_{d+n} = \begin{cases}  \xi_{d+n'} & [d] \subseteq \im f'\\ 0 & [d] \not\subseteq \im f',\end{cases}\]
because $f'$ splits into two actions---one on the $\xi_d$ part of the concatenation, and one on the $x_{d-m+1} \dots x_{d-m+n}$ part---and the action on $\xi_d$ gives zero for every $f \in \OI(d-1, d)$.

We turn our attention to the case $\ell < d$.  The element $\xi_d$ corresponds to the map $\varphi \colon i^*_d \Xi(\ell) \to i^*_d \Xi(m)$, but it also defines a map $\varphi' \colon i^*_d\Xi(d) \to i^*_d\Xi(m)$ because $f \xi_d = 0$ for every $f \in \OI(d-1, d)$, as the map is in $\ker(\res^d_{d-1})$.  Running the case $\ell = d$ on the map $\varphi'$ gives an extended map $\Xi(d) \to \Xi(m)$.  Precompose this map with the map $u_{\ell,d}\colon \Xi(\ell) \to \Xi(d)$ from \autoref{lem:XiellXid}.  This composite sends $\xi_{d,\ell}$ to $\xi_d$, and so restricts to the original map $\varphi$. 
\end{proof}

\begin{defn}
Let $M,M',M''$ be $\mathcal C$-modules, then $\Hom_{\mathcal C}^M(M',M'')$ shall denote the subgroup of maps $M'\to M''$ that factor through $M$.
\end{defn}

\begin{prop}\label{prop:factoring} For $d\ge \ell,m$, there is a short exact sequence
\[ 0 \longrightarrow \Hom_{\FI^{\op}}^{\Xi(d+1)}(\Xi(\ell),\Xi(m)) \longrightarrow \Hom_{\FI^{\op}}(\Xi(\ell),\Xi(m)) \longrightarrow  \Hom_{\FI_{\le d}^{\op}}(i^*_d\Xi(\ell),i_d^*\Xi(m)) \longrightarrow 0.\]
\end{prop}

\begin{proof}
Let $\varphi_d \colon \Xi(\ell) \to \Xi(m)$ be a map that restricts to zero in degrees $\le d$. Let $\psi_{d+1} \colon \Xi(\ell) \to \Xi(m)$ be the extension of $\res_{d+1} \varphi_d \colon i^*_{d+1} \Xi(\ell) \to i^*_{d+1}\Xi(m)$ from the proof of \autoref{thm:FJ}. By construction, $\psi_{d+1}$ factors through $\Xi(d+1)$. Let $\varphi_{d+1} = \varphi_d - \psi_{d+1}$, which now restricts to zero in degrees $\le d+1$. Iterating this procedure, we obtain a sequence of maps $\varphi_{d+i}, \psi_{d+i} \colon \Xi(\ell) \to \Xi(m)$ and equations 
\[\varphi_{d+i+1} = \varphi_{d+i} - \psi_{d+i+1}.\]
Restricting the infinite sum
\[ \psi_{d+1} + \psi_{d+2} + \cdots \]
 to any finite set of degrees in $\FI^\op$, it is a finite sum that agrees with $\varphi_d$. Because every $\psi_{d+i}$ factors through $\Xi(d+1)$, we obtain the result.
\end{proof}

\subsection{Combinatorial description of $\FJ$}\label{sec:combFJ}

In this section, we give a combinatorial description of $ \FJ(\ell,m) = \Hom_{\FI^{\op}}(\Xi(\ell),\Xi(m))$ using Lie brackets. By \autoref{thm:FJ}, we have a tower
\[ \FJ(\ell,m) \surject \dots \surject  \FJ_{\le d}(\ell,m) \surject \dots \surject  \FJ_{\le \max(\ell,m)}(\ell,m), \]
and we will find bases for the factors
\[ \ker \Big(\FJ_{\le d}(\ell,m) \surject \FJ_{\le d-1}(\ell,m) \Big) \cong \begin{cases} (\Omega^m D)_d & d-1\ge \ell,m\\ 0 &\text{otherwise}\end{cases}\]
that lift to a basis of $\FJ(\ell,m)$.

Let us start by describing $D_n\subset \Xi(0)_n$. For a finite set $S$, let $A(S)$ be the free associative $\Z$-algebra on the alphabet $S$. For $s\in S$, let $\varepsilon_s \colon A(S) \to A(S\setminus s)$ be the map defined sending $s$ to the empty word and $t$ to $t$ for all other $t\in S$. The intersection 
\[LP(S) = \bigcap_{s\in S} \ker(\varepsilon_s)\]
 is generated by all products of iterated Lie brackets which contain all elements of $S$; see Miller--Wilson \cite[Section 2.3]{MillerWilson16}. An example of an element in $LP([3])$ is
 \[ [[1,2],3] + [3,2][1,3] = (123-213-312+321) + (3213-2313-3231+2331).\]
 Identify 
 \[ \bigoplus_{\ell = 0}^{|S|} \Z\FI([\ell], S)\]
with the subgroup generated by injective words in $A(S)$. Then, $\Z\FI(n,n) \cap LP([n]) = \Z\SG_n \cap LP([n])$ is generated by products of iterated Lie brackets in which every element of $[n]$ appears exactly once. Moreover, a basis of $\Z\SG_n \cap LP([n])$ has the following description. For $S\subseteq [n]$ with $|S| \geq 2$, let $L(S)$ denote the set of Lie brackets
\[ L(S) = \{ [[\dots[[s_1, s_2], s_3], \dots ], s_{|S|}] \; \mid \text{$s_1,\dots, s_{|S|} \in S$, $\ s_i \neq s_j$ if $i \neq j$, and $s_1 = \min S$}\}. \]
Since the elements $s_2, \dots, s_{|S|}$ may be permuted, $L(S)$ has $(|S|-1)!$ elements.
Then the set 
\[ \bigsqcup_{\substack{ S_1 \sqcup \dots \sqcup S_k = [n] \\ |S_i| \ge 2, k\ge 0\\ \min S_1 < \dots < \min S_k}} L(S_1)  \cdots L(S_k) \subset A([n])\]
of products of Lie brackets gives a basis of $\Z\SG_n \cap LP([n])$. See Miller--Wilson \cite[Section 2.3]{MillerWilson16} for a more detailed treatment. As an example, this basis of $\Z\SG_4\cap LP([4])$ is
\begin{gather*}[[[1,2],3],4],\, [[[1,2],4],3],\, [[[1,3],2],4],\, [[[1,3],4],2],\,[[[1,4],2],3],\,[[[1,4],3],2],\, \\ [1,2][3,4],\, [1,3][2,4],\,[1,4][2,3].\end{gather*}
An easy combinatorial bijection using cycle decompositions shows that the cardinality of this set equals the number of derangements in $\SG_n$.

Let 
\[  \Z\SG_n \xrightarrow{(-)^{{-1}}} \Z\SG_n\]
be the linear map that inverts basis elements. We claim that the image of
\[ D_n \subset \Z\SG_n \xrightarrow{(-)^{{-1}}} \Z\SG_n \subset A([n])\]
is precisely $\Z\SG_n \cap LP([n])$. Recall that by \autoref{lem:adjoint_kernel}, $\xi \in \Xi(0)_n$ is in $D_n$ if $\delta_i \cdot \xi = 0 $ for all $i\in[n]$, where $\delta_i\in \OI(n-1,n)$ such that $i\not\in \im \delta_i$. 
For example, 
\[([[1,2],3])^{-1} = (123-213-312+321)^{-1} = 123-213-231+321\]
and 
\[ \xi = (123-213-231+321) \cdot x_1x_2x_3 = x_1x_2x_3 - x_2x_1x_3 - x_2x_3x_1 + x_3x_2x_1,\]
which drops to zero under the action of $\delta_i$:
\begin{align*}
\delta_1 \cdot \xi &= x_1x_2 - x_1x_2 - x_2x_1+x_2x_1 = 0\\
\delta_2 \cdot \xi &= x_1x_2 - x_1x_2 - x_2x_1+x_2x_1 = 0\\
\delta_3 \cdot \xi &= x_1x_2 - x_2x_1 - x_1x_2+x_2x_1 = 0.
\end{align*}
We can connect $\delta_i$ and $\varepsilon_i$ by observing that $\sigma \xi_{d,0} \in \ker(\delta_i)$ if and only if $\sigma^{-1} \in \ker(\varepsilon_i)$, which proves our claim.

Elements $\xi\in D_d\subset \Xi(0)_d$ correspond to maps in $\Hom_{\FI_{\le d}^{\op}}(i^*_d\Xi(d),i_d^*\Xi(0))$ by where the map sends $\xi_{d,d}$. In \autoref{thm:FJ}, we extended these to maps in $\Hom_{\FI^{\op}}(\Xi(d),\Xi(0))$, by sending $\xi_{n,d}$ to $\xi$ concatenated with $x_{d+1}\dots x_n$. So the bracket $[[1,2],3]$ results in the map  $\Hom_{\FI^{\op}}(\Xi(3),\Xi(0))$ that in degree $5$ sends $\xi_{5,3} = 123x_1x_2$ to
\[  x_1x_2x_3x_4x_5 - x_2x_1x_3x_4x_5 - x_2x_3x_1x_4x_5 + x_3x_2x_1x_4x_5.\]

Let us continue by describing $(\Omega^m D)_n \subset \Xi(m)_n$. By the isomorphism in \autoref{prop:Xi(ell)description}, an element
\[ (f,\sigma \xi_{n-m,0}) \in \bigoplus_{f\in\FI(m,n)} \Xi(0)_{[n]\setminus \im f}\]
corresponds the element $\tau \xi_{n,m} \in \Xi(m)_n$, where $\tau \in \SG_n$ is defined by
\[ \tau(i) = \begin{cases} \sigma(i) & i \in [n] \setminus \im f\\ f^{-1}(i) & i \in \im f.\end{cases}\]
Restricting the inverse of $\tau$, we have $(\tau^{-1})|_{[m]} = f$ and $(\tau^{-1})|^{[n]\setminus \im f}_{[n]\setminus [m]} = \sigma$. Thus, if $\sigma \in D_{[n]\setminus \im f}$ corresponds to a basis element given by a product of Lie brackets, the map
\[ \Z\SG_n \xrightarrow{(-)^{-1}} \Z\SG_n \subset A([n]),\]
sends $\tau$ to the word given by $f$ concatenated with the product of Lie brackets given by $\sigma$. For example, 
\[ 31[2,4] = 3124-3142\]
corresponds to the element
\[ (3124-3142)^{-1} \cdot 12x_1x_2 = 2x_11x_2 -2x_21x_1 \in \Xi(2)_4.\]

As before, elements $\xi\in (\Omega^mD)_d\subset \Xi(0)_d$ correspond to maps in $\Hom_{\FI_{\le d}^{\op}}(i^*_d\Xi(d),i_d^*\Xi(m))$ by where the map sends $\xi_{d,d}$. In \autoref{thm:FJ}, we extend these to maps in $\Hom_{\FI^{\op}}(\Xi(d),\Xi(m))$, by sending $\xi_{n,d}$ to $\xi$ concatenated with $x_{d-m+1}\dots x_{n-m}$. So the product $31[2,4]$ results in the map  $\Hom_{\FI^{\op}}(\Xi(4),\Xi(2))$ that in degree $6$ sends $\xi_{6,4} = 1234x_1x_2$ to
\[  2x_11x_2x_3x_4 -2x_21x_1x_3x_4.\]

Next, we describe the maps $u_{\ell,d} \colon \Xi(\ell) \to \Xi(d)$ for $\ell \le d$ from \autoref{lem:XiellXid}.  
The formula for $u_{\ell, d}$ given there is equivalent to
\[
u_{\ell, d}(\xi_{n,\ell}) = (1 \dots \ell) \cdot ((\ell+1) \dots d \; \shuffle \; x_1 \dots x_{n-d}).
\]
Recall that the shuffle product of $w_1\dots w_k$ and $w_1' \dots w_{k'}'$ is the sum
\[
w_1\dots w_k \; \shuffle \; w_1' \dots w_{k'}' = \sum_{\substack{\sigma \in \SG_{k+k'} \\ i < j \leq k \implies \sigma(i) < \sigma(j) \\ k < i < j \implies \sigma(i) < \sigma(j)}} \sigma \cdot (w_1\dots w_k w_1' \dots w_{k'}').
\]
For example, the map $u_{3,5}\colon \Xi(3) \to \Xi(5)$ sends $\xi_{7,3} = 123x_1x_2x_3x_4$ to
\[ 12345x_1x_2 + 1234x_15x_2 + 123x_145x_2 + 1234x_1x_25 + 123x_14x_25 + 123x_1x_245.\]
As a final step, by precomposing the map  in $\Hom_{\FI^{\op}}(\Xi(d),\Xi(m))$ corresponding to $\xi\in (\Omega^mD)_d$ (described above) with $u_{\ell,d} \colon \Xi(\ell) \to \Xi(d)$, we find that $\xi$ equally-well corresponds to a map in $\Hom_{\FI^{\op}}(\Xi(\ell),\Xi(m))$ if $d\ge \ell,m$.

In summary,
\[ \bigsqcup_{d \ge \ell,m} \bigsqcup_{f \in \FI(m,d)} \bigsqcup_{\substack{ S_1 \sqcup \dots \sqcup S_k = [d]\setminus \im f \\ |S_i| \ge 2, k\ge 0\\ \min S_1 < \dots < \min S_k}} L(S_1)  \cdots  L(S_k)\]
indexes a  basis of the free abelian group  $\Hom_{\FI^{\op}}(\Xi(\ell),\Xi(m))$.
We now describe the composition law in terms of this basis. Any such basis element determines a sequence $(\xi_n)$ with $\xi_n\in \Xi(m)_n$ and therefore a sequence  $(\sigma_n)$ with $\sigma_n \in \Z\SG_n$ by requiring $\xi_n = \sigma_n \xi_{n,m}$. If a second basis element corresponds to a similar sequence $(\tau_n)$ with $\tau_n\in \Z\SG_n$, then the composition corresponds to the sequence $(\sigma_n\tau_n)$.

\subsection{Exactness of tail invariants}

In this section we prove that the $\FI^\op$-modules $\Xi(\ell)$ are flat. We do so by induction employing the following proposition.  It also has another useful corollary on the tail invariants of polynomial $\FI$-modules.
  
\begin{prop} \label{prop:gammases}

There is a natural short exact sequence
\[ 0 \longrightarrow M \longrightarrow \Z\FI(-, [1] \sqcup - ) \otimes_\FI M \longrightarrow \Omega M \longrightarrow 0\]
for all $M \in \FIopmod$, and this sequence is split if $M = \Xi(\ell)$.
\end{prop}

\begin{proof}
We will use the fact that
\[ \FI(S,[1] \sqcup T) \cong \FI(S,T) \sqcup \bigsqcup_{s \in S} \FI(S \setminus s,T),\]
where an element $f\in \FI(S,[1] \sqcup T)$ corresponds to $f|^T$ if $1\not\in \im f$ and to $(f^{-1}(1), f|_{S\setminus f^{-1}(1)}^T)$ if $1 \in \im f$.
Note that there is an injection
\[ \Z\FI(S,T) \longrightarrow \Z\FI(S,[1]\sqcup T)\]
that is natural in the pair $(S, T) \in (\FI^\op \times \FI)$, and that this map sends basis elements to basis elements; similarly
\[ \Z \FI(S,[1]\sqcup T) \longrightarrow  \bigoplus_{i\in S} \Z \FI(S \setminus i,T)\]
gives a surjection of $(\FI^\op \times \FI)$-modules. 
This yields a short exact sequence
\[ 0 \longrightarrow \Z\FI(-,-) \longrightarrow \Z\FI(-,[1]\sqcup -) \longrightarrow \Omega\Z\FI(-,-) \longrightarrow 0.\]
Considered as an $\FI$-module, $\Omega\Z\FI(-,-)$ is projective, in fact, representable on the sets $S \setminus s$.  Consequently, tensoring this short exact sequence with $M$ gives us
\[ 0 \longrightarrow M \longrightarrow \Z\FI(-, [1] \sqcup - ) \otimes_\FI M \longrightarrow \Omega\Z\FI(-,-) \otimes_\FI M \longrightarrow 0.\]
We conclude the proof of the first assertion by observing
\begin{align*} \Omega\Z\FI(-,-) \otimes_\FI M 
\cong \Omega M.
\end{align*}

To prove the second assertion, set $M = \Xi(\ell)$, so that the surjection takes the form
\[\Z\FI(S,[1] \sqcup -) \otimes_\FI \Xi(\ell) \longrightarrow (\Omega\Xi(\ell))_S \cong \bigoplus_{s\in S} \Xi(\ell)_{S\setminus s},\]
where $f\otimes \xi \in \Z\FI(S,[1] \sqcup T) \otimes \Xi(\ell)_T$ is sent to zero if $1\not\in\im f$ and to \[(f^{-1}(1), f|_{S\setminus f^{-1}(1)}^T \cdot \xi)\] if $1\in \im f$.
We will now construct a natural section
\[ \bigoplus_{s\in S} \Xi(\ell)_{S\setminus s} \longrightarrow \Z\FI(S,[1] \sqcup -) \otimes_\FI \Xi(\ell).\]
Let $\phi \colon S \setminus s \to [\ell] \sqcup \{x_1, \dots, x_{|S|-\ell - 1}\}$ for real numbers $x_1 <  \dots <  x_{|S|-\ell - 1}$.  We send $(s,[\phi])$ to
\[ (f_s \otimes [\phi]) - \big((S\subset [1] \sqcup S) \otimes [\phi_s]\big),\]
where $f_s\in \FI(S,[1] \sqcup(S\setminus s))$ is given by $f_s(s) = 1$ and by $S\setminus s \subset [1] \sqcup(S\setminus s )$ on the remaining elements; and $\phi_s$ is given by $\phi_s(s) = x_{|S|-\ell}$ with  $x_{|S|-\ell} > x_{|S|-\ell - 1}$, and by $\phi_s|_{S\setminus s} = \phi$.

To see that this assignment gives a section for each $S$, note that $(f_s)|_{S\setminus f^{-1}(1)}^{S\setminus s} = \id_{S \setminus s}$, and so
$f_s \otimes [\phi]$ is sent to 
\[ ( f_s^{-1}(1) ,  (f_s)|_{S\setminus f^{-1}(1)}^{S\setminus s} \cdot [\phi]) = (s, [\phi]),\]
and $\big((S\subset [1] \sqcup S) \otimes [\phi_s]\big)$ is sent to zero, as this inclusion does not have $1$ in its image.

It remains to prove that this defines a map of $\FI^\op$-modules. Let $g\in \FI(S',S)$, then
\[ g\cdot (s,[\phi]) = \begin{cases} (s', g|_{S'\setminus s'}\cdot [\phi] )& s\in  \im g\\ 0 &\text{otherwise,}\end{cases}\]
where $g(s') = s$. 
Via the proposed section, this is sent to
\[\begin{cases} (f'_{s'} \otimes g|_{S'\setminus s'}\cdot [\phi]) - \big((S'\subset [1]\sqcup S') \otimes[\phi'_{s'}]\big)& s \in\im g\\ 0 &\text{otherwise,}\end{cases}\]
where $f'_{s'}$ satisfies $f'_{s'}(s') = 1$ and $f'_{s'}|_{S' \setminus s'}$ equals the inclusion $S'\setminus s' \subset [1]\sqcup(S'\setminus s')$; and $\phi'_{s'}|_{S'\setminus s'} = \phi'$ and  $\phi'_{s'}(s') =x_{|S'|-\ell}$, where $\phi'\colon S'\setminus s' \to [\ell] \sqcup \{x_1, \dots, x_{|S'|-\ell-1}\}$ such that  $[\phi'] = g|_{S'\setminus s'}\cdot [\phi]$.

Acting by $g$ happens in the first tensor factor:
\begin{align*}
g\Big( (f_s \otimes [\phi]) - \big((S\subset [1]\sqcup S) \otimes [\phi_s]\big)\Big) = (gf_s \otimes [\phi]) - \big(g(S\subset [1]\sqcup S) \otimes [\phi_s]\big).
\end{align*}
If $g(s') = s$,
\[ gf_s = f'_{s'} \cdot g|_{S'\setminus s'},\quad g(S\subset [1]\sqcup S) = (S\subset [1]\sqcup S)([1]\sqcup g)\quad\text{and}\quad [\phi'_{s'}] =  g[\phi_s].\]
This proves naturality in the case that $s\in \im g$.  On the other hand, if $s\not\in \im g$, 
\[ g(S \subset [1]\sqcup S) = gf_s \big([1]\sqcup (S\setminus s) \subset [1] \sqcup S\big)\quad\text{and}\quad \big((S\setminus s) \subset S\big)\cdot [\phi_s] = [\phi].\]
Then,
\begin{align*}
g(S\subset [1]\sqcup S) \otimes [\phi_s] &=  gf_s \big([1]\sqcup (S\setminus s) \subset [1] \sqcup S\big) \otimes [\phi_s] \\
&=  gf_s \otimes \big((S\setminus s) \subset S\big) \cdot [\phi_s] \\
&= gf_s \otimes [\phi].
\end{align*}
This proves that 
\[ g\Big( (f_s \otimes [\phi]) - \big((S\subset [1]\sqcup S) \otimes [\phi_s]\big)\Big)  = 0\]
in this case. Naturality follows.
\end{proof}

Let us give a quick corollary about polynomial functors.

\begin{cor}\label{cor:polytailinvariants}
Let $M$ be an $\FI$-module presented in finite degree and assume that it has polynomial degree $\le d$. Then $M \otimes_\FI \Xi(\ell) \cong 0$ if $\ell> d$.
\end{cor}

\begin{proof}
We prove this corollary by induction over the polynomial degree. For the purposes of this proof, we consider an $\FI$-module to be polynomial degree $\le -1$, if it is eventually zero. This is consistent with the induction because $M$ has polynomial degree $\le 0$ if and only if $\coker(M \to  \Sigma M)$ is eventually zero. 

Assume first $M$ is eventually zero. Thus by \autoref{thm:main_concrete}, $M\otimes_\FI \Xi(\ell) \cong 0$ for all $\ell >-1$.  This establishes the base case $d=-1$.

Let us now assume that $M$ has polynomial degree $\le d$ for some $d\ge 0$, so that we can assume $\coker(M \to \Sigma M)\otimes_\FI \Xi(\ell) \cong 0$ for all $\ell> d-1$ by induction. We now connect the tail invariants of $\coker(M \to \Sigma M)$ to those of $M$. 
Consider the right-exact sequences
\[ M \longrightarrow \Sigma M \longrightarrow \coker(M \to \Sigma M) \longrightarrow 0\]
and the (actually short exact) sequence from \autoref{prop:gammases}
\[ \Xi(\ell) \longrightarrow \Z\FI(-,[1]\sqcup -) \otimes_\FI \Xi(\ell) \longrightarrow \Omega \Xi(\ell) \longrightarrow 0,\]
where the $\Omega\Xi(\ell) \cong \Xi(\ell +1)$. Tensoring the first right-exact sequence by $\Xi(\ell)$, we get that
\[ M \otimes_\FI \Xi(\ell) \longrightarrow \Sigma M \otimes_\FI \Xi(\ell) \longrightarrow \coker(M \to \Sigma M)\otimes_\FI \Xi(\ell) \longrightarrow 0\]
is exact. Tensoring the second right-exact sequence by $M$, we get that
\[ M \otimes_\FI \Xi(\ell) \longrightarrow \Sigma M \otimes_\FI \Xi(\ell) \longrightarrow M\otimes_\FI \Xi(\ell+1) \longrightarrow 0\]
is exact. Consequently, 
\[  \coker(M \to \Sigma M)\otimes_\FI \Xi(\ell) \cong M \otimes_\FI \Xi(\ell+1).\]
This implies the assertion that $M \otimes \Xi(\ell+1) \cong 0$ for $\ell >d-1$.
\end{proof}

To prove flatness of $\Xi(0)$, we use a classical result of Isbell, which was originally of interest only as a counterexample.
\begin{thm}[\cite{Isbell74}] \label{thm:isbell}
The functor $\colim_{\OI} \colon \xmod{\OI} \to \Ab$ is exact.
\end{thm}

\begin{cor} \label{cor:Xiflat}
The $\FI^\op$-module $\Xi(0)$ is flat.
\end{cor}
\begin{proof}
By \autoref{prop:XiKan}, tensoring with $\Xi(0)$ is the same as restricting to $\OI$ and tensoring with $\Lambda(0)$.  But $\Lambda(0) = \underline{\mathbb{Z}}$, and tensoring with this constant functor this is the same as taking a colimit.
\end{proof}
We now leverage  \autoref{cor:Xiflat} to prove that the $\FI^\op$-modules $\Xi(\ell)$ are flat for all $\ell \in \mathbb{N}$.  This method of proof has been employed in a similar context by Gan--Li \cite{coinduction}.

\begin{prop}\label{prop:Xiflat}
For every $\ell \in \N$, the $\FI^\op$-module $\Xi(\ell)$ is flat.
\end{prop}
\begin{proof}
We employ induction on $\ell$.  The base case, $\ell = 0$, is \autoref{cor:Xiflat}.
By \autoref{prop:gammases}, $\Xi(\ell + 1)$ is a summand of $\Z\FI(-,[1]\sqcup -) \otimes_\FI \Xi(\ell)$.  However, tensoring with $\Z\FI(-,[1]\sqcup -) \otimes_\FI \Xi(\ell)$
 is the same as precomposing with $([1] \sqcup -)$ and then tensoring with $\Xi(\ell)$; both steps are exact---here we use the inductive hypothesis that $\Xi(\ell)$ is flat---and so the tensor product is flat as well.  This shows flatness, as $\Xi(\ell + 1)$ is a summand of a flat module.
\end{proof}


\section{Tails of $\FI$-modules via a new basis}\label{sec:CB}
\subsection{The Catalan Basis}
We define a new basis for the module $\Z \FI(k, n)$.  
\begin{defn} \label{defn:CBl}
Let
\[ \Catalan(\ell,n) = \{ c\colon [\ell] \to [n] \mid c(i) \ge 2i \text{ for all $i\in [\ell]$}\}.\]
Let $k,\ell,n\in\N$ and $c \in \Catalan(\ell, n)$. Define the set
\[ \CB^c_\ell(k,n) = \pi_0(\{\phi \in \Emb([k], \im c \sqcup \R) \mid  \im c \subseteq \im \phi\}),\]
where we have used the subspace topology of $\Emb$.  Let 
\begin{gather*} \CB_\ell(k,n) = \bigsqcup_{c\in \Catalan(\ell,n)} \CB_\ell^c(k,n),\\
\CB(k,n) = \bigsqcup_{\ell\in\N} \CB_\ell(k,n) ,\\
 \CB'(k,n) = \bigsqcup_{\ell=0}^{\min(k,n-k)} \CB_\ell(k,n).\end{gather*}
\end{defn}

\begin{lem} \label{lem:CB'}
$\CB'(k,n) \subseteq \CB(k,n)$ is an equality if $n\ge 2k-1$.
\end{lem}

\begin{proof}
We need to prove that $\CB_\ell(k,n) = \emptyset$ if $\ell \ge \min(k,n-k)+1$.  In \autoref{defn:CBl}, we require of any element in $\CB_\ell(k,n)$ that $\im c \subseteq \im \phi$, which implies $\ell \leq k$, and so $\CB_\ell(k,n)$ is empty if $\min(k,n-k) = k$.  Supposing instead that $\min(k,n-k) = n-k$, then $k \leq n-k+1 \leq \ell \leq k$, and so $\ell = k$ and $n = 2k-1$.  In this case, $\Catalan(\ell, n)$ is empty, and so $\CB_\ell(k,n)$ is again empty.
\end{proof}

Note that $\Z\CB_\ell^c(k,n) \cong H_0(X, A)$, where $X = \Emb([k], \im c \sqcup \mathbb{R})$ and $A$ is the subspace of embeddings $\phi \colon [k] \to\im c \sqcup \R$ with $\im c \not\subseteq \im \phi$, so that $\Z\CB_\ell^c(-,n)$ becomes an $\FI^\op$-module in a manner analogous to \autoref{defn:Xi(ell)}.  The resulting module is isomorphic to $\Xi(\ell)$, since the only difference is a relabeling $[\ell] \cong \im c$. Let us denote this isomorphism by
\begin{equation}\label{eq:CBXiiso}
\kappa_c \colon \Xi(\ell) \xrightarrow{\cong} \Z\CB^c_\ell(-,n).
\end{equation}
We will omit the subscript $c$ if it is clear from the context.

\begin{notation}\label{rem:CBdescription}
As in \autoref{rem:x}, the set $\CB_\ell^c(k,n)$, whose elements are homotopy classes, is in bijection with the set of representatives
\[\{ \phi\colon [k] \inject \im c \sqcup \{ x_1, \dots, x_{k-\ell}\}\}.\]
If we write such a $\phi$ in one-line notation, it denotes the corresponding element in $\CB_\ell^c(k,n)$.
\end{notation}

\begin{prop} \label{prop:count}
For all $k, n \in \mathbb{N}$, we have \[|\FI(k,n)| = |\CB'(k,n)|.\] In particular, $|\FI(k,n)| = |\CB(k,n)|$ if $n\ge 2k-1$.
\end{prop}

\begin{proof}
By the description in \autoref{rem:CBdescription}, $|\CB_\ell(k,n)| = k! \cdot |\Catalan(\ell,n)|$ for all $\ell \le k$. Because $|\Catalan(\ell,n)| = \binom{n}{\ell} - \binom{n}{\ell -1}$ for $\ell \le m := \min(k,n-k)$,
\begin{multline*} \sum_{\ell = 0}^{\min(k,n-k)}|\CB_\ell(k,n)| = k!  \cdot \left( \binom{n}{0} + \left[ \binom{n}{1} - \binom{n}{0} \right] + \cdots + \left[ \binom{n}{m} - \binom{n}{m-1} \right] \right)\\
 = k! \cdot \binom{n}{m}  = \; k! \cdot \binom{n}{k} =  |\FI(k,n)|.\end{multline*}
 The second assertion follows from \autoref{lem:CB'}.
\end{proof}

\subsection{A perfect pairing} \label{sec:pairing}
\begin{defn}
Let $f \in \FI(k,n)$, and let $\phi\colon [k] \inject \im c \sqcup \R$ for some $c\in \Catalan(\ell,n)$ be a representative of $\cb \in \CB(k,n)$.  We say  $f$  \emph{matches}  $\cb$ if $f(i) = \phi(i)$ when $\phi(i) \in \im c$ and $f(i) < f(j) \iff \phi(i) <\phi(j)$ when $\phi(i),\phi(j) \in \R$.

Define a bilinear form
\[\langle -,-\rangle \colon \mathbb{Z} \FI(k,n) \otimes \mathbb{Z}{\CB}(k, n) \to \mathbb{Z} \]
by its values on pairs of basis vectors
\[
\langle f, \cb \rangle = \begin{cases}
1 & \text{if $f$ matches $\cb$} \\
0 & \text{otherwise. }
\end{cases}
\]
\end{defn}

\begin{rem}
The idea of this pairing is that an element $\cb\in \CB_l^c(k,n)$ is a template for injections $[k] \to [n]$. In this template, the values of $\cb$ in $\im c$ must match exactly, but for the values in $\R$, only the order has to match. Using the bijection from \autoref{rem:CBdescription}, suppose $\cb = x_152x_2 \in  \CB_2(4,5)$.  
Since the $x_i$ can be replaced by any numbers with $x_1<x_2 $, the following injections match $\cb$: $1523, 1524, 3524 \in \FI(4,5)$.
\end{rem}

\begin{prop}\label{prop:X}
Fix $n\in \N$. There is a homomorphism of $\FI^\op$-modules
\[\X\colon \Z\FI(-,n) \longrightarrow \Z\CB(-,n)\]
sending $f\in \FI(k,n)$ to 
\[ \sum_{\cb \in \CB(k,n)} \langle f,\cb\rangle\cdot \cb.\]
\end{prop}

\begin{proof}
Fix $f\in \FI(k,n)$ and $g\in \FI(k',k)$. We want to show that $\X(gf) = g\X(f)$. Recall that $\cb \in \CB(k,n)$ is sent to an element $g\cb \in \CB(k',n)$ by precomposing. Let $\Pi_f$ be the set of $\cb \in \CB(k,n)$ that match $f$. It is enough to show that 
\[ \{ g\cb \mid \cb \in \Pi_f \}\]
is the set of elements in $\CB(k',n)$ that match $gf$. This follows from the soon-to-be-given \autoref{prop:casework}, \eqref{prop:casework3} and the easy-to-check fact that $\sigma f$ matches $\sigma \cb$ if and only if $f$ matches $\cb$ for all $\sigma \in \SG_k$.
\end{proof}

We intend to prove the following theorem.

\begin{thm} \label{thm:perfect}
The restricted paring $\langle -,-\rangle\colon \Z\FI(k,n) \otimes \Z\CB'(k,n) \to \Z$ is perfect. In particular, the pairing $\langle -,-\rangle\colon \Z\FI(k,n) \otimes \Z\CB(k,n) \to \Z$ is perfect if $n\ge 2k-1$.
\end{thm}

We prepare some notation to aid the proof of \autoref{thm:perfect}.   
Let
\[ \eps \colon \FI(k,n) \longrightarrow \FI(k,n+1) \]
be postcomposition by $[n] \subset [n+1]$. Similarly, let
\[ e \colon \CB_\ell(k,n) \longrightarrow \CB_\ell(k,n+1) \]
be the map that sends the $\sqcup$-summand indexed by $c$ to the one indexed by the composition of $c$ and $[n]\subset [n+1]$.  
For example, $e(x_125x_2) =x_125x_2$.  Let
\[ s\colon \CB_\ell(k,n) \longrightarrow \CB_{\ell+1}(k,n+1)\]
be the map that replaces $x_{k-\ell}$ by $(n+1)$ when $\ell \le \min(k-1, n-k)$. (Otherwise the map is not defined.) For example, $s(x_125x_2) = x_1257$, if $n=6$. Let
\[ \tau_j \colon \FI(k,n) \longrightarrow \FI(k+1,n+1)\]
be the map that sends $f\in \FI(k,n)$ to
\[ \tau_j(f) (i) = \begin{cases} f(i) &\text{if $i<j$,}\\n+1& \text{if $i =j$,}\\f(i-1)&\text{if $i>j$.}\end{cases}\]
Similarly, let
\[t_j \colon \CB_\ell(k,n) \longrightarrow \CB_\ell(k+1,n)\]
be the map that sends $\cb \in \CB_\ell(k,n)$ to
\[ \tau_j(\cb) (i) = \begin{cases} \cb(i) &\text{if $i<j$,}\\x_{k-\ell +1}& \text{if $i =j$,}\\ \cb(i-1)&\text{if $i>j$.}\end{cases}\]
For example, $\tau_4(x_125x_2) = x_125x_3x_2$.

\begin{prop} \label{prop:casework}
We have the following elementary properties of the previously defined operations $e,s,t_j,\varepsilon, \tau_j$:
\begin{enumerate}
\item $\tau_jf$ matches $e\cb \Longleftrightarrow \tau_j f$ matches $s\cb$ \label{prop:casework1}
\item $\varepsilon f$ never matches $s\cb$
\item $\varepsilon f$ matches $e\cb \Longleftrightarrow f$ matches $\cb$ \label{prop:casework3}
\item $\tau_jf$ matches $et_{j'}\cb \Longleftrightarrow$ $(f \mbox{ matches }\cb)$ and $(j=j')$.
\end{enumerate}
\end{prop}
\begin{proof}
\ 
\begin{enumerate}
\item If $f\in \FI(k-1,n-1)$ and $\cb \in \CB_\ell(k,n-1)$, then $e\cb$ and $s\cb$ coincide in all positions except $x_{k-\ell}$ is replaced by $n$. In order for $\tau_j f$ to match $e\cb$ or $s\cb$, that position must be the $j$th in both cases. 
\item Let $f\in \FI(k,n-1)$ and $\cb \in \CB_\ell(k,n-1)$. Then the image of $\eps f$ does not contain $n$, but it would have to in order to match $s \cb$.
\item Clear.
\item Let $f\in \FI(k-1,n-1)$ and $\cb \in \CB_\ell(k-1,n-1)$. ``$\Longleftarrow$'' is clear because the ``new'' $x_{k-\ell}$ in $j$th position of $t_j \cb$ matches the ``new'' $n$ in the $j$th position of $\tau_j f$. For ``$\Longrightarrow$'', observe that $n$ is not in the image of $et_{j'}\cb$, so the largest $x$ must be in the $j$th position. Thus $j=j'$. For the other positions to match, we must have that $f$ matches $\cb$. \qedhere
\end{enumerate}

\end{proof}
Define $r \colon \mathbb{Z}\CB(k,n) \to \mathbb{Z}\CB(k,n+1)$ by sending $\cb \in \CB_\ell(k,n)$ to
\[ r \cb =
\begin{cases}
s \cb & \text{if $l \leq k-1$} \\
0 & \text{otherwise.}
\end{cases}\]

\begin{prop} \label{prop:r}
$\langle \tau_j f, e \cb \rangle = \langle \tau_j f, r \cb \rangle$ for all $f \in \FI(k,n)$ and $\cb \in \CB'(k+1,n)$.
\end{prop}
\begin{proof}
If $\cb \in \CB_\ell(k+1, n)$ with $\ell \leq k$, then $r\cb = s\cb$, and so \autoref{prop:casework}, \eqref{prop:casework3} gives the result.  In the remaining case, $\ell = k + 1$ and $\langle \tau_j f, r \cb \rangle = 0$.  Then there are no $x$'s in the image of $e \cb$, and moreover, $n+1$ does not appear either. Therefore the $n+1$ in $\tau_j f$ doesn't match any position in $e\cb$ and so $\langle \tau_j f, e\cb\rangle = 0$.
\end{proof}
 
The functions $\varepsilon$ and $\tau_1, \ldots, \tau_{k+1}$ induce a bijection
\begin{equation}\label{eq:decomposition}
\FI(k+1,n) \sqcup \bigsqcup_{j\in [k+1]} \FI(k,n) \longrightarrow \FI(k+1,n+1)
\end{equation}
sorting injections $g \in \FI(k+1,n+1)$ according to the preimage $g^{-1}(\{n+1\})$, which is empty, or a singleton $\{j\}$ for some $j \in [k+1]$.

\begin{proof}[Proof of  \autoref{thm:perfect}]
We show by induction on $(k,n)$ that the functions  
\[
\left\{ \; \langle -, \cb \rangle \colon \FI(k,n) \to \mathbb{Z} \; | \; \cb \in \CB'(k,n) \right\}
\]
span the full space of functions $\FI(k,n) \to \mathbb{Z}$;  this proves that the restricted pairing $\langle -, -\rangle$ is perfect by  \autoref{prop:count}.

There are two base cases: $k=0$ and $n=0$.  When $k=0$, the sets $\FI(0,n)$ and $\CB(0,n)$ are singletons.  Note that those two single elements match.  This is a perfect pairing.  On the other hand, if $n=0$ and $k>0$, then the sets $\FI(k,0)$ and $\CB(k,0)$ are empty, so $\langle -, - \rangle$ is perfect vacuously.

We proceed to the inductive step.  Let $\lambda \colon \FI(k+1, n+1) \to \mathbb{Z}$ be an arbitrary function.  By the inductive hypothesis, for each $j \in [k+1]$, the function $f \mapsto \lambda (\tau_j f)$ has an expansion in the basis $\{\;\langle -, \omega \rangle \; | \; \omega \in \CB'(k,n) \}$.  In other words, there exist integers $\alpha_j^{\omega}$ so that, for all $f \in \FI(k,n)$,
\[
\lambda (\tau_j f) = \sum_{\omega \in \CB'(k, n)} \alpha^{\omega}_j \cdot \langle f, \omega \rangle.
\]
Similarly, there exist integers $\beta^{\cb}$ so that, for all $f \in \FI(k+1,n)$,
\[
\lambda(\varepsilon f) -  \left(\sum_{\substack{1 \leq j \leq k+1  \\ \omega \in \CB'(k, n)}} \alpha^{\omega}_j \cdot \langle \varepsilon f, t_j \omega \rangle \right) = \sum_{\rho \in \CB'(k+1, n)} \beta^{\rho} \cdot \langle f, \rho \rangle.
\]
Set $m = \min \{ k, n-k \}$.  We claim
\begin{equation} \label{eq:claim}
\lambda = \left( \sum_{j=1}^{k+1} \sum_{\omega} \alpha_j^{\omega} \langle -, t_j \omega \rangle \right) + \left( \sum_{\rho}  \beta^{\rho} \langle -, e \rho \rangle \right) - \left( \sum_{l=0}^{m} \sum_{\nu}   \beta^{\nu} \langle -, s \nu \rangle \right),
\end{equation}
with sums ranging over $\omega \in \CB'(k,n)$, $\rho \in \CB'(k+1, n)$, and $\nu \in \CB'_\ell(k+1,n)$.  We prove the claim \eqref{eq:claim} using the decomposition \eqref{eq:decomposition}, showing equality after evaluation at injections in the image of each map $\varepsilon$ and $\tau_j$.

Suppose $f \in \FI(k+1,n)$.  Compute
\begin{align*}
\lambda(\varepsilon f) &= \left(\lambda(\varepsilon f) - \sum_{\rho } \beta^{\rho} \langle f, \rho \rangle \right) + \left( \sum_{\rho}  \beta^{\rho} \langle  f, \rho \rangle \right) \\
&= \left(\lambda(\varepsilon f) - \sum_{\rho } \beta^{\rho} \langle f, \rho \rangle \right) + \left( \sum_{\rho}  \beta^{\rho} \langle \varepsilon f, e \rho \rangle \right) - \left( \sum_{l=0}^{z} \sum_{\nu}   \beta^{\nu} \langle \varepsilon f, s \nu \rangle \right) \\
&= \left( \sum_{j=1}^{k+1} \sum_{\omega} \alpha_j^{\omega} \langle \varepsilon f, t_j \omega \rangle \right) + \left( \sum_{\rho}  \beta^{\rho} \langle \varepsilon f, e \rho \rangle \right) - \left( \sum_{l=0}^{z} \sum_{\nu}   \beta^{\nu} \langle \varepsilon f, s \nu \rangle \right),
\end{align*}
where we have used that $\langle f, \rho \rangle = \langle \varepsilon f, e \rho \rangle$ and $\langle \varepsilon f, s \nu \rangle = 0$ by \autoref{prop:casework}. 

Now suppose $f \in \FI(k,n)$ and $j' \in [k+1]$.  Compute
\begin{align*}
\lambda(\tau_{j'} f) &=  \left( \sum_{\omega} \alpha_{j'}^{\omega} \langle f,  \omega \rangle \right) \\
&= \left( \sum_{\omega} \alpha_{j'}^{\omega} \langle f,  \omega \rangle \right) + \left( \sum_{\rho}  \beta^{\rho} \left( \langle \tau_{j'} f, e \rho \rangle - \langle \tau_{j'} f, r \rho \rangle\right) \right) \\
&= \left( \sum_{\omega} \alpha_{j'}^{\omega} \langle f,  \omega \rangle \right) + \left( \sum_{\rho}  \beta^{\rho} \langle \tau_{j'} f, e \rho \rangle \right) - \left( \sum_{\rho}   \beta^{\rho} \langle \tau_{j'} f, r \rho \rangle \right) \\
&= \left( \sum_{\omega} \alpha_{j'}^{\omega} \langle f,  \omega \rangle \right) + \left( \sum_{\rho}  \beta^{\rho} \langle \tau_{j'} f, e \rho \rangle \right) - \left( \sum_{l=0}^{z} \sum_{\nu}   \beta^{\nu} \langle \tau_{j'} f, s \nu \rangle \right) \\
&= \left( \sum_{j=1}^{k+1} \sum_{\omega} \alpha_j^{\omega} \langle \tau_{j'} f, t_j \omega \rangle \right) + \left( \sum_{\rho}  \beta^{\rho} \langle \tau_{j'} f, e \rho \rangle \right) - \left( \sum_{l=0}^{z} \sum_{\nu}   \beta^{\nu} \langle \tau_{j'} f, s \nu \rangle \right),
\end{align*}
where $\langle \tau_{j'} f, e \rho \rangle - \langle \tau_{j'} f, r \rho \rangle=0$ by  \autoref{prop:r}, and the sum over $j$ has all summands equal to zero, apart from the one where $j=j'$ by  \autoref{prop:casework}.  Thus, \eqref{eq:claim} holds, and so $\lambda$ is in the span of functions of the form $\langle-,\cb\rangle$.
\end{proof}

\begin{cor}\label{cor:X}
Fix $n\in \N$. The homomorphism of $\FI^\op$-modules
\[\X\colon \Z\FI(-,n) \longrightarrow \Z\CB(-,n)\]
from \autoref{prop:X} induces an isomorphism
\[\Z\FI(k,n) \stackrel{\cong}\longrightarrow \Z\CB(k,n)\]
if $n\ge 2k-1$. And 
\[\xymatrix{
\Z\CB(k,n) \ar[r]^{\X^{-1}} \ar[d]_{f\cdot} &\Z\FI(k,n)\ar[d]_{f\cdot} \\
\Z\CB(\ell,n) \ar[r]^{\X^{-1}} & \Z\FI(\ell,n)
}\]
commutes for every $f\in \FI(\ell,k)$ when  $n \ge 2k-1$.
\end{cor}

\subsection{Proofs of \hyperref[thm:main_concrete]{Theorems \ref{thm:main_concrete}} and \ref{thm:compute}}

\begin{proof}[Proof of  \autoref{thm:main_concrete}]
In the proof of this theorem, we make use of the formula $A_\ell = M \otimes_\FI \Xi(\ell)$ from the introduction. Assume $M$ is presented in degrees $\leq d$ so that
\[ M_n \cong (i_d^*M) \otimes_{\FI_{\le d}} \Z\FI(i_d-,n)\]
for all $n\in \N$. By \autoref{cor:X}, 
\[  \Z\FI(i_d-,n) \cong \Z\CB(i_d-,n)\]
for all $n\ge 2d-1$.  Arguing directly,
\begin{align*}
M_n &\cong (i_d^*M) \otimes_{\FI_{\le d}} \Z\FI(i_d-,n) \\
&\cong (i_d^*M) \otimes_{\FI_{\le d}} \Z\CB(i_d-,n) \\
&\cong (i_d^*M) \otimes_{\FI_{\le d}} \left(\bigoplus_{\ell=0}^d (i_d^*\Xi(\ell))^{\oplus \Catalan(\ell,n)} \right)\\
&\cong (i_d^*M) \otimes_{\FI_{\le d}} \Z\FI(i_d-,-) \otimes_\FI\left(\bigoplus_{\ell=0}^d (\Xi(\ell))^{\oplus \Catalan(\ell,n)} \right)\\
&\cong M \otimes_{\FI} \left(\bigoplus_{\ell=0}^d \Xi(\ell)^{\oplus \Catalan(\ell,n)} \right)\\
&\cong \bigoplus_{\ell=0}^d \left(M \otimes_{\FI} \Xi(\ell)\right)^{\oplus \Catalan(\ell,n)}\\
&\cong \bigoplus_{\ell=0}^d \left(M \otimes_{\FI} \Xi(\ell)\right)^{\oplus \binom n\ell - \binom n{\ell-1}}\qedhere
\end{align*}
\end{proof}

\begin{proof}[Proof of  \autoref{thm:compute}]
The $\FI$-matrix $Z$ describes a map
\[
\left( \bigoplus_{j} \Z\FI(b_j,-) \right) \xrightarrow{Z} \left(\bigoplus_{i} \Z\FI(a_i, -)\right).
\]
To compute this map after application of the functor $- \otimes_{\FI} \Xi(\ell)$, we use Yoneda's lemma on the entries.  Suppose $f \colon a \to b$ is a morphism of $\FI$, appearing in one of the entries of $Z$.  Let $\Z\FI(f,-)$ be the map on free $\FI$-modules induced by precomposition with $f$, and compute
\begin{align*}
& \left[\Z\FI(b,-) \xrightarrow{\Z\FI(f,-)} \Z\FI(a,-) \right] \otimes_{\FI} \Xi(\ell)\\
=& \left[\Z\FI(b,-)\otimes_{\FI} \Xi(\ell) \xrightarrow{\Z\FI(f,-)\otimes_{\FI} \Xi(\ell)} \Z\FI(a,-)\otimes_{\FI} \Xi(\ell) \right] \\
=& \left[\Xi(\ell)_b \xrightarrow{f\cdot} \Xi(\ell)_a \right] \\
=& \; \Xi(\ell)_f.
\end{align*}
  Extending this computation to $\mathbb{Z}$-linear combinations, and to block matrices, we obtain the formula $Z \otimes_{\FI} \Xi(\ell) = \Xi(\ell)_Z$.  The result follows, since tensor products preserve cokernels. 
\end{proof}


\section{The category of $\FI$-tails}
In this section, we reconstruct the tail of an $\FI$-module $M$ presented in finite degree from the tensor products $M \otimes_\FI \Xi(\ell)$ for $\ell \in \mathbb{N}$.

Recall from \autoref{defn:FJ}, $\FJ$ is the category whose objects are $\{0,1,2,\dots\}$ and $\FJ(\ell,m) = \Hom_{\FI^{\op}}(\Xi(\ell),\Xi(m))$, and from \autoref{defn:FJd}, $\FJ_{\le d}$ is  the category whose objects are $\{0,1,2,\dots,d\}$ and $\FJ_{\le d}(\ell,m) = \Hom_{\FI^{\op}_{\le d}}(i^*_d\Xi(\ell),i^*_d\Xi(m))$.

\begin{defn}\label{defn:pidell}
Let us define some elements in $\Xi(\ell)_d$ and, by the isomorphism $\kappa_c$ from \eqref{eq:CBXiiso}, corresponding elements in $\CB^{c}_\ell(d,2d)$ where $c = 24 \dots (2\ell) \in \Catalan(\ell,2d)$. These element assist in the proof of \autoref{thm:main}. Let
\[ \xi_{d,\ell} = 1 \dots \ell x_1 \dots x_{d-\ell} \in \Xi(\ell)_d \quad \text{and}\quad \cb_{d,\ell} = \kappa_c(\xi_{d,\ell}) = 24 \dots (2\ell) x_1 \dots x_{d-\ell} \in \CB^c_\ell(d,2d),\]and let
\begin{multline*} \zeta_{d,\ell} = x_1 1 \dots x_\ell \ell x_{\ell+1} \dots x_{d-\ell} \in  \Xi(\ell)_d \quad \text{and}\\ \rho_{d,\ell} = \kappa_c(\zeta_{d,\ell})= x_1 2 \dots x_\ell (2\ell) x_{\ell+1} \dots x_{d-\ell}  \in \CB^c_\ell(d,2d).\end{multline*}
Let $\tau_{d,\ell} \in \SG_{d}$ denote the unique permutation such that $\tau_{d,\ell} \xi_{d,\ell} = \zeta_{d,\ell}$ and $\tau_{d,\ell} \cb_{d,\ell}= \rho_{d,\ell}$.
\end{defn}

\begin{lem}\label{lem:a1b2*24=12}
$\X^{-1}(\cb_{d,\ell})\zeta_{2d,\ell} = \xi_{d,\ell}$ and $\X^{-1}(\cb_{d,\ell})\rho_{2d,\ell} = \cb_{d,\ell}$.
\end{lem}

\begin{proof}
For every $f \in \FI(d,2d)$ that matches $\cb_{d,\ell}$, we get $f\zeta_{2d,\ell} = \xi_{d,\ell}$.  
 Further, if $f \in \FI(d,2d)$ matches $\sigma \cb_{d,\ell}$ for some $\sigma \in \SG_d$, then $\sigma^{-1} f$ matches $\cb_{d,\ell}$, and $f\zeta_{2d,\ell}= \sigma \xi_{d,\ell}$.  If $f \in \FI(d, 2d)$ fails to match any such  $\sigma \cb_{d,\ell}$, it must be because $\{2,4,\dots,2\ell\}$ is not contained in $\im f$.  In this case, $f\zeta_{2d,\ell} = 0$.
  
Therefore, if $\X^{-1}(\cb_{d,\ell}) = \sum_{f\in\FI(d,2d)} a_f f\in \Z \FI(d,2d)$, then
\begin{align*}
 \X^{-1}(\cb_{d,\ell})\zeta_{2d,\ell} &= \sum_{f\in\FI(d,2d)} a_f f \zeta_{2d,\ell} \\
 &=  \sum_{\substack{f\in\FI(d,2d) \\ \{2, 4, \ldots, 2\ell\} \not \subseteq \im f}} a_f f \zeta_{2d,\ell} +  \sum_{\sigma \in \SG_d} \sum_{\substack{f\in\FI(d,2d) \\ \langle f,\sigma\cb_{d,\ell}\rangle = 1}} a_f f \zeta_{2d,\ell}\\
 &=  \hspace{25pt} 0   \hspace{25pt} +   \hspace{52pt} \sum_{\sigma \in \SG_d} \sum_{\substack{f\in\FI(d,2d) \\ \langle f,\sigma\cb_{d,\ell}\rangle = 1}} a_f \sigma \xi_{d,\ell} \\
  &=   \sum_{\sigma \in \SG_d} \sum_{f\in\FI(d,2d)} \langle f,\sigma\cb_{d,\ell}\rangle \cdot a_f  \sigma \xi_{d,\ell} \\
 &=   \sum_{\sigma \in \SG_d} \sum_{f\in\FI(d,2d)} \langle a_f f,\sigma\cb_{d,\ell}\rangle \cdot \sigma \xi_{d,\ell} \\
  &=   \sum_{\sigma \in \SG_d} \langle \X^{-1}(\cb_{d,\ell}) ,\sigma\cb_{d,\ell}\rangle \cdot \sigma \xi_{d,\ell} \\
 \end{align*}
Notice that for $\omega \in \CB(d,2d)$
\[ \langle \X^{-1}(\cb_{d,\ell}), \omega \rangle = \begin{cases} 1 & \cb_{d,\ell}= \omega\\ 0 &\cb_{d,\ell} \neq \omega.\end{cases}\]
Thus 
\[ \X^{-1}(\cb_{d,\ell})\zeta_{2d,\ell} = \sum_{\sigma \in \SG_d} \langle \X^{-1}(\cb_{d,\ell}),\sigma\cb_{d,\ell} \rangle\sigma\xi_{d,\ell} = \xi_{d,\ell}.\qedhere\]
\end{proof}

\begin{defn}
For $\ell \le d$, let $\Theta_{\le d}(\ell) \subseteq \Z\FI(d,-)$ be the $\FI$-submodule that is generated by $\X^{-1}(\cb_{d,\ell}) \in \Z\FI(d,2d)$.
\end{defn}

We include two lemmas about $\Theta_{\le d}(\ell)$ for later use.

\begin{lem}\label{lem:thetapres}
The $\FI$-module $\Theta_{\le d}(\ell)$ is presented in degrees $\le 3d+1$.
\end{lem}

\begin{proof}
By definition, $\Theta_{\le d}(\ell)$ is generated in degree $2d$. Let $Q$ be the quotient of $\Z\FI(d,-)$ modulo $\Theta_{\le d}(\ell)$. Clearly, $Q$ is generated in degree $d$ and presented in degree $2d$. By \cite[Theorem A]{CE17}, the next syzygies of $Q$ are generated in degrees $\le 3d+1$. Thus $\Theta_{\le d}(\ell)$ is presented in degrees $\le 3d+1$.
\end{proof}

\begin{lem}\label{lem:thetapoly}
The $\FI$-module $\Theta_{\le d}(\ell)$ has polynomial degree $\le d$.
\end{lem}

\begin{proof}
Subquotients of $\FI$-modules with polynomial degree $\le d$ also have polynomial degree $\le d$. (See e.g.\ \cite[2.8(c)]{NS17}.) Therefore the assertion follows because $\Z\FI(d,-)$ has polynomial degree $\le d$.
\end{proof}

\begin{prop}\label{prop:Thetahomcovariant}
$\Theta_{\le d}(\ell) \otimes_{\FI}\Xi(-) \cong \FJ_{\le d}(\ell, -)$
\end{prop}

\begin{proof}
Since $\Xi(m)$ is flat by \autoref{prop:Xiflat}, the map
\[ \alpha_m \colon \Theta_{\le d}(\ell) \otimes_{\FI} \Xi(m) \subseteq \Z\FI(d,-) \otimes_{\FI} \Xi(m) \stackrel{\cong}{\longrightarrow} \Xi(m)_d\]
is an injection and $\alpha_m(f\otimes x) = f\cdot x$.
Further,  \autoref{prop:mapseq} states that there is an injection
\[ \beta_m \colon \Hom_{\FI^{\op}_{\le d}}(i^*_d\Xi(\ell), i^*_d\Xi(m)) \inject \Xi(m)_d,\]
where a map $\varphi$ is sent to $\varphi(\xi_{d,\ell})$. Let $\psi \in \Hom_{\FI^{\op}_{\le d}}(i^*_d\Xi(m), i^*_d\Xi(n))$ and let $\widetilde \psi \in \Hom_{\FI^\op}(\Xi(m),\Xi(n))$ be an arbitrary lift of $\psi$, whose existence is guaranteed by \autoref{thm:FJ}. Then the squares
\begin{equation}\label{eq:commdiag}
\xymatrix{
\Theta_{\le d}(\ell) \otimes_{\FI}\Xi(m)  \ar[r]^>>>>>{\alpha_m} \ar[d]^{\id \otimes \widetilde\psi} & \Xi(m)_d \ar[d]^{\psi_d} & \Hom_{\FI^{\op}_{\le d}}(i^*_d\Xi(\ell), i^*_d\Xi(m))\ar[l]_>>>>>{\beta_m}\ar[d]^{\psi \circ }\\
\Theta_{\le d}(\ell) \otimes_{\FI}\Xi(n) \ar[r]^>>>>>>{\alpha_n} & \Xi(n)_d & \Hom_{\FI^{\op}_{\le d}}(i^*_d\Xi(\ell), i^*_d\Xi(n))\ar[l]_>>>>>>{\beta_n}
}\end{equation}
commute because 
\[\psi_d( \alpha_m( f\otimes \xi)) = \psi_d(f\xi) = f \widetilde\psi_k(\xi) = \alpha_n\circ (\id \otimes \widetilde \psi) (f \otimes \xi) \]  for every $f \in \FI(d,k)$ and $\xi \in \Xi(m)_k$, and 
\[\beta_n(\psi_*(\varphi)) =\beta_n( \psi \circ \varphi) =  (\psi\circ \varphi)(\xi_{d,\ell}) = \psi_d( \varphi(\xi_{d,\ell})) = \psi_d(\beta_m(\xi_{d,\ell})) \]
for every $\varphi \in \Hom_{\FI^{\op}_{\le d}}(i^*_d\Xi(\ell), i^*_d\Xi(m))$.
These commuting squares show that both $\Theta_{\le d}(\ell) \otimes_{\FI} \Xi(-)$ and $\Hom_{\FI^{\op}_{\le d}}(i^*_d \Xi(\ell), i^*_d \Xi(-))$ are $\FJ_{\le d}$-submodules of $\Xi(-)_d$.
Thus, we want to prove that $\Theta_{\le d}(\ell) \otimes_{\FI}\Xi(m) $ and $\Hom_{\FI^{\op}_{\le d}}(i^*_d\Xi(\ell), i^*_d\Xi(m))$ have the same image in $\Xi(m)_d$.

In order to prove $\im \beta_m \subseteq \im \alpha_m$, we need to show that $\beta_m(\varphi)\in \im \alpha_m$ for all homomorphisms $\varphi \in \Hom_{\FI^\op_{\le d}}(i^*_d\Xi(\ell), i^*_d \Xi(m))$. Let  $\widetilde \varphi \in \Hom_{\FI^\op}(\Xi(m),\Xi(n))$ be a lift of such a $\varphi$.
Using \autoref{lem:a1b2*24=12}, and that $\widetilde \varphi$ commutes with the action of $\FI^\op$, we get
\[ \varphi(\xi_{d,\ell}) = \varphi( \X^{-1}(\cb_{d,\ell}) \cdot \zeta_{2d, \ell}) =  \X^{-1}(\cb_{d,\ell}) \cdot \widetilde \varphi( \zeta_{2d,\ell}) \] 
Recall that $\X^{-1}(\cb_{d,\ell})$ is the generator of $\Theta_{\le d}(\ell)$; thus $\beta_m(\varphi) = \varphi(\xi_{d,\ell})$ is in the image of $\alpha_m$.

We now argue the reverse inclusion. 
First observe that since $\Theta_{\le d}(\ell)$ is generated by $\X^{-1}(\cb_{d,\ell})$, the tensor product $\Theta_{\le d}(\ell) \otimes_\FI \Xi(m)$ is generated by the elements $\X^{-1}(\cb_{d,\ell}) \otimes \xi$ for $\xi \in \Xi(m)_{2d} $. For a $\xi \in \Xi(m)_{2d}$, consider the following composition
\[\varphi\colon i^*_d\Xi(\ell) \xrightarrow{\xi_{d,\ell} \mapsto \cb_{d,\ell}} \Z\CB^{24\dots(2\ell)}_\ell(i_d - ,2d) \xrightarrow{\X^{-1}} \Z\FI(i_d -, 2d)  \xrightarrow{f\mapsto f \cdot \xi} i^*_d \Xi(m).\] 
Then
\[\alpha_m( \X^{-1}(\cb_{d,\ell}) \otimes \xi) =  \X^{-1}(\cb_{d,\ell}) \cdot \xi= \varphi(\xi_{d,\ell}) =  \beta_m(\varphi)\]
is in the image of $\beta_m$.
\end{proof}

\begin{prop}\label{prop:ThetaFJopmodule}
For every 
\[\varphi \in \Hom_{\FJ}(\FJ_{\le d}(m,-) , \FJ_{\le d}(\ell, -)),\]
there is a  homomorphism 
$\widetilde\psi \colon\Z\FI(2d,-) \to \Z\FI(d,-)$
that restricts to a map 
\[\psi \colon \Theta_{\le 2d}(m) \to \Theta_{\le d}(\ell).\] 
 Invoking \autoref{prop:Thetahomcovariant}, $\psi$ induces a homomorphism of $\FJ$-modules 
\[ \psi \otimes_\FI \Xi(-) \colon \FJ_{\le 2d}(m,-) \to \FJ_{\le d}(\ell, -)\]
that factors through $\varphi$ by
\[ \FJ_{\le 2d}(m,-) \surject \FJ_{\le d}(m, - ) \xrightarrow{\varphi} \FJ_{\le d}(\ell,-),\]
where $\FJ_{\le 2d}(m,-) \surject \FJ_{\le d}(m, - )$ is the restriction as in \autoref{thm:FJ}.
\end{prop}

\begin{proof}
Set $\gamma = \varphi(\id_m)\in \FJ_{\le d}(\ell,m)$. Recall that $\gamma\colon i_d^*\Xi(\ell) \to i^*_d\Xi(m)$ is an $\FI_{\le d}^\op$-module homomorphism  by definition. Let $ \gamma^\kappa$ denote $\gamma$ conjugated by the natural isomorphism $\kappa$:
\[  \gamma^\kappa\colon i^*_d\CB_\ell^{24\dots (2\ell)}(-,2d) \xrightarrow{\kappa^{-1}} i^*_d \Xi(\ell) \xrightarrow{\gamma} i^*_d\Xi(m) \xrightarrow{\kappa} i^*_d\CB_m^{24\dots (2m)}(-,2d)\]
We define $\widetilde\psi\colon \Z\FI(2d,-) \to \Z\FI(d,-)$ by precomposing with
\[(\X^{-1}\circ \gamma^\kappa)(\cb_{d,\ell})\cdot \tau_{2d,\ell} \in \Z\FI(d,2d).\] 
It remains to prove that $\widetilde \psi$ restricts to 
\[ \psi \colon \Theta_{\le 2d}(m) \to \Theta_{\le d}(\ell)\]
and $\im  (\psi \otimes_\FI \Xi(-)) = \im \varphi$.

The first assertion is that precomposing with $(\X^{-1}\circ \gamma^\kappa)(\cb_{d,\ell})\cdot \tau_{2d,\ell}$ sends $\Theta_{\le 2d}(m)$ into $\Theta_{\le d}(\ell)$. In fact, by \autoref{lem:a1b2*24=12}, $\cb_{d,\ell} = \X^{-1}(\cb_{d,\ell}) \cdot \rho_{2d,\ell}$, and so 
\[
(\X^{-1}\circ \gamma^\kappa)(\cb_{d,\ell}) = (\X^{-1}\circ \gamma^\kappa)(\X^{-1}(\cb_{d,\ell})\cdot \rho_{2d,\ell}) = \X^{-1}(\cb_{d,\ell})\cdot (\X^{-1}\circ \gamma^\kappa)(\rho_{2d,\ell}),\]
where the second equality comes from the fact that $(\X^{-1} \circ  \gamma^\kappa)$ commutes with $\X^{-1}(\cb_{d,\ell}) \in \Z\FI(d,2d)$ by \autoref{cor:X}. This proves the first assertion because $\Theta_{\le d}(\ell)$ is generated by $\X^{-1}(\cb_{d,\ell})$.

For the second assertion, we have to understand the map
\[ \psi\otimes_\FI \Xi(-) \colon \FJ_{\le 2d}(m,-) \to \FJ_{\le d}(\ell,-).\] 
Consider the following commutative diagram.
\[\xymatrix{
\Z\FI(2d,-) \otimes_\FI \Xi(-)  \ar[rr]^{\widetilde \psi \otimes_\FI \Xi(-)} &&\Z\FI(d,-) \otimes_\FI \Xi(-) \\
\Theta_{\le 2d}(m) \otimes_\FI \Xi(-)  \ar@{_(->}[u]^{} \ar[rr]^{\psi \otimes_\FI \Xi(-)} &&\Theta_{\le d}(\ell)\otimes_\FI \Xi(-)\ar@{_(->}[u]^{}
}\]
Note that the upwards maps are injective because $\Xi(n)$ is a flat $\FI^\op$-module for every $n\in \N$ by \autoref{prop:Xiflat}. This diagram simplifies using Yoneda's lemma and \autoref{prop:Thetahomcovariant} to the left square of the following diagram.
\[\xymatrix{
\Xi(-)_{2d} \ar[r]^{} &\Xi(-)_{d} \\
 \FJ_{\le 2d}(m,-) \ar@{_(->}[u]^{} \ar[r]^{} &\FJ_{\le d}(\ell,-) \ar@{_(->}[u]^{} & \ar[l]_{\varphi}  \FJ_{\le d}(m,-)
}\]
We want to show that $\id_{m} \in \FJ_{\le2d}(m,m)$ is sent to $\varphi(\id_m) \in \FJ_{\le d}(\ell,m)$. The result will follow because $\id_m \in \FJ_{\le 2d}(m,m)$ maps to $\id_m \in \FJ_{\le d}(m,m)$ by the restriction.
We can show this by proving that these two identity maps are sent to the same element in $\Xi(m)_d$ in the top row.

The identity map in $\FJ_{\le 2d}(m,m)$ is sent to $\xi_{2d, m} \in \Xi(m)_{2d}$, which then is mapped to $(\X^{-1}\circ  \gamma^\kappa)(\cb_{d,\ell})\tau_{2d,\ell} \cdot\xi_{2d,m} \in \Xi(m)_d$ by the definition of $\widetilde \psi$. On the other hand, $\varphi(\id_m) = \gamma \in \FJ_{\le d}(\ell,m)$, which is sent to $\gamma(\xi_{d,\ell}) \in \Xi(m)_d$.

Let us start proving these two elements agree by rewriting the first as
\[(\X^{-1}\circ  \gamma^\kappa)(\cb_{d,\ell})\tau_{2d,\ell} \xi_{2d,m} = (\X^{-1}\circ  \gamma^\kappa)(\cb_{d,\ell}) \zeta_{2d,\ell} \]
using \autoref{defn:pidell}. Since $\Xi(m)_d$ is a free left $\Z\SG_d$-module generated by $\xi_{d,m}$, there exist coefficients $a_{\sigma} \in \Z$ such that 
\[ \gamma(\xi_{d,\ell}) = \sum_{\sigma \in \SG_d} a_\sigma \sigma \xi_{d,m},\]
and thus
\[  \gamma^\kappa(\cb_{d,\ell}) = \sum_{\sigma \in \SG_d} a_\sigma \sigma \cb_{d,m}.\]
Then
\begin{multline*}
(\X^{-1}\circ  \gamma^\kappa)(\cb_{d,\ell})\zeta_{2d,\ell} =\X^{-1}(\sum_{\sigma \in \SG_d} a_\sigma \sigma \cb_{d,m})\zeta_{2d,\ell} \\=\sum_{\sigma \in \SG_d} a_\sigma \sigma \X^{-1}( \cb_{d,m})\zeta_{2d,\ell} = \sum_{\sigma \in \SG_d} a_\sigma \sigma \xi_{d,m} = \varphi(\xi_{d,\ell})\end{multline*}
by \autoref{lem:a1b2*24=12}.
\end{proof}

We are now ready to prove \hyperref[thm:main]{Theorems \ref{thm:main}} and \ref{thm:degrees}.

\begin{proof}[Proof of {\hyperref[thm:main]{Theorems \ref{thm:main}}} and \ref{thm:degrees}]
Let us consider the functor $\xmod{\FI} \to \xmod{\FJ}$ that sends an $\FI$-module $M$ to the $\FJ$-module $M \otimes_{\FI} \Xi(-)$. If we restrict this functor to $\FI$-modules that are presented in finite degree, then we can limit the codomain to the category of $\FJ$-modules supported in finite degree because $M \otimes_{\FI} \Xi(\ell) \cong 0$ for $\ell$ larger than the presentation degree of $M$. This functor factors through the category of $\FI$-tails because it is exact by \autoref{prop:Xiflat} and annihilates $\FI$-modules supported in finite degree by \autoref{thm:main_concrete}. We will show that the induced functor is an equivalence of categories.

Let us start with essential surjectivity. 
Let $N$ be an $\FJ$-module supported on $\{0, \dots, d\}$. Then $N$ is the cokernel of 
\[  \varphi \colon \bigoplus_i \FJ_{\le d}(m_i, - ) \longrightarrow  \bigoplus_j \FJ_{\le d}(\ell_j, -).\]
Consider the map
\[ \psi \colon \bigoplus_i \Theta_{\le 2d}(m_i) \longrightarrow  \bigoplus_j \Theta_{\le d}(\ell_j)\]
corresponding to $\varphi$ via \autoref{prop:ThetaFJopmodule}. Let $M = \coker\psi$. 
By \autoref{prop:Thetahomcovariant} and \autoref{prop:ThetaFJopmodule}, 
\[ \coker(\psi \otimes_\FI \Xi(-)) \cong N.\] Because tensoring is right exact, this cokernel coincides with $M \otimes_\FI \Xi(-)$.

We now prove faithfulness. Suppose $\varphi\colon M \to M'$ is a map of $\FI$-modules presented in finite degree, which induces the zero map $M \otimes_\FI \Xi(-) \to M' \otimes_\FI \Xi(-)$. We need to show that $\im \varphi$ has finite support. Say $M$ and $M'$ are presented in degree $\le d$. Then \autoref{thm:main_concrete} implies that 
\[ M_n \cong \bigoplus_{\ell = 0}^d (M \otimes_\FI \Xi(\ell))^{\oplus \Catalan(\ell,n)} \quad\text{and}\quad M'_n \cong \bigoplus_{\ell = 0}^d (M' \otimes_\FI \Xi(\ell))^{\oplus \Catalan(\ell,n)}\]
for all $n\ge 2d-1$. Thus $\varphi_n\colon M_n \to M'_n$ is the zero map for all $n\ge 2d-1$, which implies that the image is supported on $\{0,\dots, 2d-2\}$.

Fullness is the statement that \[\Hom_{\FI}(M,M') \to \Hom_{\FJ}(M \otimes_{\FI} \Xi(-), M' \otimes_\FI \Xi(-))\] is surjective. Fix a map of $\FJ$-modules
\[ F \colon M \otimes_{\FI} \Xi(-)\longrightarrow M' \otimes_\FI \Xi(-).\]
Because both $M$ and $M'$ are supported in finite degree, we can consider them as $\FJ_{\le d}$-modules for $d\gg 0$. Choose presentations
\[ M \otimes_{\FI} \Xi(-) \cong \coker\Big( \bigoplus_i \FJ_{\le d}(m_i, - ) \longrightarrow  \bigoplus_j \FJ_{\le d}(\ell_j, -)\Big)\]
and
\[ M' \otimes_{\FI} \Xi(-) \cong \coker\Big( \bigoplus_{i'} \FJ_{\le d}(m'_{i'}, - ) \longrightarrow  \bigoplus_{j'} \FJ_{\le d}(\ell'_{j'}, -)\Big)\]
of $M$ and $M'$ as $\FJ_{\le d}$-modules.
Expanding $F$ in terms of generators and relations, we find the commutative diagram
\[ \xymatrix{
 \bigoplus_i \FJ_{\le d}(m_i, - )  \ar[r]^{\varphi_3}\ar[d]^{\varphi_4}& \bigoplus_j \FJ_{\le d}(\ell_j, -) \ar[r]\ar[d]^{\varphi_2} & M \otimes_\FI \Xi(-) \ar[d]^{F}\ar[r] & 0\\
  \bigoplus_{i'} \FJ_{\le d}(m'_{i'}, - )  \ar[r]^{\varphi_1} &  \bigoplus_{j'} \FJ_{\le d}(\ell'_{j'}, -)\ar[r] & M'\otimes_\FI \Xi(-)\ar[r] &0.
}\]
We want to construct a commutative square
\begin{equation}\label{eq:thetasquare}
\xymatrix{
\bigoplus_i \Theta_{\le 4d}(m_i )\ar[r]^{\psi_3}\ar[d]^{\psi_4}& \bigoplus_j \Theta_{\le 2d}(\ell_j)\ar[d]^{\psi_2}\\
\bigoplus_{i'} \Theta_{\le 2d}(m'_{i'} ) \ar[r]^{\psi_1}&  \bigoplus_{j'} \Theta_{\le d}(\ell'_{j'})
}
\end{equation}
such that $\coker(\psi_3)$ and $\coker(\psi_1)$ have the same tail as $M$ and $M'$, respectively. 

To construct $\psi_1$ and $\psi_2$, we invoke \autoref{prop:ThetaFJopmodule} on $\varphi_1$ and $\varphi_2$, respectively. Then $\coker(\psi_1)$ has the same tail as $M'$ as in the proof of essential surjectivity. Next, let us construct
\[ \xymatrix{
 \bigoplus_k \FJ_{\le 2d}(n_k, - )  \ar[r]^{\widetilde\varphi_3}\ar[d]^\vartheta& \bigoplus_j \FJ_{\le 2d}(\ell_j, -) \ar[r]\ar@{->>}[d]^{\res^{2d}_d} & M \otimes_\FI \Xi(-) \ar@{=}[d]\ar[r] & 0\\
  \bigoplus_{i} \FJ_{\le d}(m_{i}, - )  \ar[r]^{\varphi_3} &  \bigoplus_{j} \FJ_{\le d}(\ell_{j}, -)\ar[r] & M\otimes_\FI \Xi(-)\ar[r] &0,
}\]
where $\widetilde\varphi_3$ is giving a presentation of $M \otimes_\FI \Xi(-)$ as an $\FJ_{\le 2d}$-module, and the map $\vartheta$ exists by projectivity of $ \bigoplus_k \FJ_{\le 2d}(n_k, - )$ as an $\FJ_{\leq 2d}$-module.  Let $\psi_3$ be the map constructed from $\widetilde \varphi_3$ using \autoref{prop:ThetaFJopmodule}. Thus, $\coker(\psi_3)$ has the same tail as $M$. Finally, let $\widetilde \varphi_4$ be a map making the diagram
\[\xymatrix{
  \bigoplus_k \FJ_{\le 2d}(n_k, - )  \ar[rr]^{\vartheta}\ar[d]^{\widetilde \varphi_4}& & \bigoplus_i \FJ_{\le d}(m_i, - )  \ar[d]^{\varphi_4}\\
\bigoplus_{i'} \FJ_{\le 2d}(m'_{i'}, - )  \ar@{->>}[rr]^{\res^{2d}_d} & &\bigoplus_{i'} \FJ_{\le d}(m'_{i'}, - ) ,
}\]
commute, using that $\bigoplus_k\FJ_{\le 2d}(n_k,-)$ is a projective $\FJ_{\le 2d}$-module. And let $\psi_4$ be constructed from $\widetilde \varphi_4$ using \autoref{prop:ThetaFJopmodule}. The constructions yield the commutative diagram
\[\resizebox{\columnwidth}{!}{
$\displaystyle\xymatrix{
\bigoplus_k \FJ_{\le 4d}(n_k, - )\ar@{->>}[rd]\ar@/_/[rddd]_{\psi_4\otimes \Xi(-)} \ar@/^/[rrrd]^{\psi_3\otimes\Xi(-)} \\
& \bigoplus_k \FJ_{\le 2d}(n_k, - )  \ar[rd]^{\vartheta}\ar[rr]^{\widetilde \varphi_3}\ar[dd]^{\widetilde \varphi_4}& &\bigoplus_j \FJ_{\le 2d}(\ell_j, -) \ar@{->>}[d] \ar[rd] \\
&&  \bigoplus_{i} \FJ_{\le d}(m_{i}, - ) \ar[d]^{\varphi_4} \ar[r]^{\varphi_3} &  \bigoplus_{j} \FJ_{\le d}(\ell_{j}, -)\ar[r] \ar[d]^{\varphi_2}& M\otimes_\FI \Xi(-)\ar[r]\ar[d]^F &0\\
&\bigoplus_{i'} \FJ_{\le 2d}(m'_{i'}, - )  \ar@{->>}[r]&\bigoplus_{i'} \FJ_{\le d}(m'_{i'}, - )\ar[r]^{\varphi_1} & \bigoplus_{j'} \FJ_{\le d}(\ell'_{j'}, -)\ar[r]& M'\otimes_\FI \Xi(-)\ar[r] &0.
}$
}\]
This implies that the square \eqref{eq:thetasquare} commutes after tensoring with $\Xi(-)$. Using the decomposition of \autoref{thm:main_concrete} and \autoref{lem:thetapres}, it follows that \eqref{eq:thetasquare} commutes in degrees $\ge 24d+1$. Let $\widetilde M$ and $\widetilde M'$ be the cokernels of $\psi_3$ and $\psi_1$, respectively. Then $\widetilde M \otimes_\FI \Xi(-) = M \otimes_\FI \Xi(-)$ and $\widetilde M' \otimes_\FI \Xi(-) = M' \otimes_\FI \Xi(-)$. For the induced map $f\colon \widetilde M \to \widetilde M'$, we get that $f \otimes \Xi(-) = F$ from the above commutative diagram.

For the second part, we need show that every $\FI$-module with polynomial degree $\le d$ is sent to a module in ${\xmod{\FJ}}$ that is supported in degrees $\{0,\dots, d\}$, which follows immediately from \autoref{cor:polytailinvariants}, and that for every $\FJ$-module supported in degrees $\{0, \dots, d\}$, there is an $\FI$-module of polynomial degree $\le d$ in its preimage, which follows  from our proof of essential surjectivity and \autoref{lem:thetapoly}.
\end{proof}

\bibliographystyle{amsalpha}
\bibliography{references}
\end{document}